\documentclass[letterpaper, 11pt]{amsart}  
\usepackage[utf8]{inputenc}
\usepackage[T1]{fontenc}
\usepackage[english]{babel}
\usepackage[centertags]{amsmath}
\usepackage{amstext,amssymb,amsopn,amsthm}
\usepackage{mathrsfs}
\usepackage{dsfont}
\usepackage{bbm}
\usepackage{thmtools}
\usepackage{graphicx}
\usepackage[titletoc,title]{appendix}

\usepackage[backgroundcolor=white, bordercolor=blue,
linecolor=blue]{todonotes}
\usepackage[colorlinks=true, linkcolor=black, citecolor=black]{hyperref}
\usepackage{enumitem}
\setlist[enumerate]{itemsep=0mm}

\addto\extrasenglish{}
\addto\extrasenglish{}
\addto\extrasenglish{}

\theoremstyle{plain}
\declaretheorem[title=Theorem, parent=section]{theorem}
\declaretheorem[title=Lemma,sibling=theorem]{lemma}

\declaretheorem[title=Corollary,sibling=theorem]{corollary}

\theoremstyle{definition}
\declaretheorem[title=Definition,sibling=theorem]{definition}
\declaretheorem[title=Remark,sibling=theorem]{remark}
\declaretheorem[title=Remark, numbered=no]{remark*}
\declaretheorem[title=Example, sibling=theorem]{example}
\declaretheorem[title=Assumption, numbered=no]{assumption*}

\numberwithin{equation}{section}

\newcommand{\N}{\mathds{N}}
\newcommand{\R}{\mathds{R}}

\def\hmath$#1${\texorpdfstring{{\rmfamily\textit{#1}}}{#1}}

\newcommand{\cE}{\mathcal{E}}

\newcommand{\eps}{\varepsilon}

\newcommand{\BIGOP}[1]
{
\mathop{\mathchoice%
{\raise-0.22em\hbox{\huge $#1$}}%
{\raise-0.05em\hbox{\Large $#1$}}{\hbox{\large $#1$}}{#1}}}

\def\Xint#1{\mathchoice
   {\XXint\displaystyle\textstyle{#1}}%
   {\XXint\textstyle\scriptstyle{#1}}%
   {\XXint\scriptstyle\scriptscriptstyle{#1}}%
   {\XXint\scriptscriptstyle\scriptscriptstyle{#1}}%
   \!\int}
\def\XXint#1#2#3{{\setbox0=\hbox{$#1{#2#3}{\int}$}
     \vcenter{\hbox{$#2#3$}}\kern-.5\wd0}}

\def\dashint{\Xint-}
\newcommand{\BIGboxplus}{\mathop{\mathchoice%
{\raise-0.35em\hbox{\huge $\boxplus$}}%
{\raise-0.15em\hbox{\Large $\boxplus$}}{\hbox{\large $\boxplus$}}{\boxplus}}}

\DeclareMathOperator{\supp}{supp}

\DeclareMathOperator{\dvg}{div}

\DeclareMathOperator{\tail}{Tail}
\DeclareMathOperator{\pv}{p.v.}
\newcommand{\U}{\widetilde{u}}

\renewcommand{\d}{\textnormal{d}}

\begin{document}
\allowdisplaybreaks
 \title{The parabolic Harnack inequality for nonlocal equations}

\author{Moritz Kassmann}
\author{Marvin Weidner}

\address{Fakult\"{a}t f\"{u}r Mathematik\\Universit\"{a}t Bielefeld\\Postfach 100131\\D-33501 Bielefeld}
\email{moritz.kassmann@uni-bielefeld.de}
\urladdr{www.math.uni-bielefeld.de/$\sim$kassmann}

\address{Departament de Matemàtiques i Informàtica, Universitat de Barcelona, Gran Via de les Corts Catalanes 585, 08007 Barcelona, Spain}
\email{mweidner@ub.edu}
\urladdr{https://sites.google.com/view/marvinweidner/}


\thanks{Moritz Kassmann gratefully acknowledges financial support by the German Research Foundation (SFB 1283 - 317210226). Marvin Weidner has received funding from the European Research Council (ERC) under the Grant Agreement No 801867 and by the AEI project PID2021-125021NA-I00 (Spain).}


\begin{abstract}
We complete the local regularity program for weak solutions to linear parabolic nonlocal equations with bounded measurable coefficients. Within the variational framework we prove the parabolic Harnack inequality and Hölder regularity estimates. We discuss in detail the shortcomings of previous results in this direction. The key element of our approach is a fine study of the nonlocal tail term.
\end{abstract}

\maketitle

\section{Introduction}  
\label{sec:intro}
The aim of this work is to prove the Harnack inequality for weak solutions to the parabolic equation
\begin{align*}
\partial_t u - L_t u = 0 ~~ \text{ in } I \times \Omega.
\end{align*}
Here, $I \subset [0,\infty)$ is a finite open interval and $\Omega \subset \R^d$ is a bounded domain. $L_t$ is a linear nonlocal operator of the form
\begin{align*}
-L_t u (x) = \pv \int_{\R^d} (u(x) - u(y)) K(t;x,y) \d y,
\end{align*}
and $K : [0,\infty) \times \R^d \times \R^d \to [0,\infty]$ is a measurable jumping kernel satisfying for some $\Lambda \ge \lambda > 0$
\begin{align}
\label{eq:Kcomp}\tag{$K_{\asymp}$}
\lambda (2-\alpha) |x-y|^{-d-\alpha} \le K(t;x,y) \le \Lambda (2-\alpha) |x-y|^{-d-\alpha}
\end{align}
for every $t \in I$ and $x,y \in \R^d$, where $\alpha \in [\alpha_0, 2)$ for some $\alpha_0 \in (0,2)$. Moreover, we assume that $K$ is symmetric, i.e.
\begin{align*}
K(t;x,y) = K(t;y,x) ~~ \forall t \in I, ~~ \forall x,y \in \R^d.
\end{align*}
Due to this symmetry condition, $L_t$ can be interpreted as a nonlocal operator in divergence form and the equation $\partial_t u - L_t u$ has to be interpreted in the weak sense. It is thus natural to use variational methods in order to analyze the solutions.

Let us explain our results. One aim of this work is to prove in \autoref{thm:fHI} the parabolic Harnack inequality for nonlocal operators with bounded measurable coefficients. Thus we complete the De Giorgi-Nash-Moser theory for parabolic equations governed by operators of the form \eqref{eq:Kcomp}. Existing versions of this result are either more restrictive or incomplete as we explain below. An interesting peculiarity is that, so far, there has been no analytic proof of this result, not even for kernels satisfying \eqref{eq:Kcomp}. It is an interesting task to establish the parabolic Harnack inequality under weaker assumptions than \eqref{eq:Kcomp}. In \autoref{sec:ext} we provide such corresponding extensions. \\
A standard application of the parabolic Harnack inequality is the derivation of parabolic Hölder regularity estimates. Previous works have proved these estimates under the unnatural additional assumption that the parabolic tail term is bounded in time. We obtain Hölder regularity estimates in \autoref{thm:HRE} under a weaker $L^{1+\eps}$-condition. As we show in a counterexample, see \autoref{ex:counterexample}, this assumption is optimal. Previous results have missed this point, which we explain below in detail. 

\subsection*{Main results}

In order to state our main result, let us introduce some helpful notation for different time-intervals. Let $t_0 \in [0,\infty)$ and $R > 0$. We define
\begin{align*}
	I_R^{\oplus}(t_0) &= (t_0, t_0+R^{\alpha}),~~ I_R^{\ominus}(t_0) = (t_0-R^{\alpha}, t_0),~~ I_R(t_0) = (t_0-R^{\alpha}, t_0+R^{\alpha}).
\end{align*}

\begin{theorem}[full Harnack inequality]
	\label{thm:fHI}
	Assume that $K$ satisfies \eqref{eq:Kcomp}. Then, for every $R > 0$, $t_0 \in I$, $x_0 \in \Omega$ with $I_{4R}(t_0) \times B_{4R}(x_0) \subset I \times \Omega$, and any globally nonnegative solution $u$ to $\partial_t u - L_t u = 0$ in $I \times \Omega$ it holds
	\begin{align}
		\label{eq:fPHI}
				\sup_{I^{\ominus}_{R}(t_0 - R^{\alpha}) \times B_R(x_0)} u \le c \inf_{I_{R}^{\oplus}(t_0) \times B_R(x_0)} u.
	\end{align}
	Here, $c = c(d,\alpha_0,\lambda,\Lambda) > 0$ is a constant.
\end{theorem}

\begin{remark}\label{rem:discussion-full-harnack}
	Let us discuss to which extent \autoref{thm:fHI} is new.
	\begin{itemize}
	\item[(i)] As a consequence of estimate \eqref{eq:fPHI} resp. of \autoref{thm:locbdL1tail},  weak solutions resp. subsolutions to $\partial_t u - L_t u = 0$ are locally bounded. This property has not been proved yet without additional assumptions. 
	\item[(ii)] In contrast to the Harnack inequality in \cite{BaLe02}, \cite{ChKu03}, \cite{CKW20}, \cite{CKW21},
	\autoref{thm:fHI} is robust as $\alpha \nearrow 2$ (see \autoref{rem:robust}). 
	\item[(iii)] Different from the probabilistic approach, we allow for time-in\-homo\-geneous jumping kernels and work with weak solutions in the natural function spaces, see \autoref{def:col-concept}.	
	\item[(iv)] An attempt to prove the parabolic Harnack inequality \eqref{eq:fPHI} for a parabolic exterior value problem with given continuous, bounded data is made in \cite{Kim19}. Lemma 5.3 and Theorem 5.4 need to be reformulated, though. They only hold for solutions $u \in H^1(I; X_g(\Omega))$, not for subsolutions. As a result, some constants in the proof like $\epsilon_0$ and $d_2$ in (5.14) depend on the data $g$. \footnote{The authors thank the author of \cite{Kim19} for discussions of this issue.} 
	\item[(v)] Different from \autoref{thm:fHI}, \cite{Str19b} established the parabolic Harnack inequality \eqref{eq:fPHI} for global solutions to $\partial_t u - L_t u = 0$ in $I \times \R^d$. The proof in \cite{Str19b} cannot be adapted to solutions on domains. 
	\item[(vi)] The statement of the Harnack inequality in \cite{KaWe22b} is different from \eqref{eq:fPHI} because \cite{KaWe22b} involves the tail of $u^+$. A special feature of \cite{KaWe22b} is that the approach allows for nonsymmetric kernels. 
	\item[(vii)] For the sake of accessibility, we state and prove the aforementioned results under the strong assumption \eqref{eq:Kcomp}. We explain possible relaxations of this condition and provide a corresponding extension of the results in \autoref{sec:ext}, see \autoref{thm:fHI_sharp-ass} and \autoref{thm:fHI_singular}.
\end{itemize}
\end{remark}

We make several remarks on possible extensions of our main result:

\begin{remark}
	Note that \autoref{thm:fHI} can be easily generalized to weak solutions $u$ to $\partial_t u - L_t u = f$, where $f \in L^{\mu,\theta}_{t,x}(I \times \Omega)$ satisfying $\frac{1}{\mu} + \frac{d}{\alpha\theta} < 1$ by using standard techniques. Moreover, it is possible to consider $f \in L^{1,\infty}_{t,x}(I \times \Omega)$. For more details, we refer to the proof of \autoref{cor:fPHI}.
\end{remark}

We extend \autoref{thm:fHI} in two ways. As shown in the proof of \autoref{thm:improvedwHI-parabolic}, one can estimate the tail of $u_+$ by the infimum of the solution $u$. Hence, one can allow for a larger left-hand side in \eqref{eq:fPHI}.  Moreover, the assumption on $u$ to be globally nonnegative can easily be weakened to $u$ nonnegative in $I \times \Omega$. As a consequence, the tail of $u_-$ needs to be added on the right-hand side in \eqref{eq:fPHI}. Given a measurable function $v:\R^d \to \R$ and a ball $B_R(x_0)$, the tail is defined as follows:
\begin{align} \label{eq:tail-def}
	\tail(v;R,x_0) = (2-\alpha) \int_{\R^d \setminus B_R(x_0)} |v(y)||x_0 - y|^{-d-\alpha} \d y \,.
\end{align}

\begin{corollary}[full Harnack inequality with tail]
	\label{cor:fPHI}
	Assume that $K$ satisfies \eqref{eq:Kcomp}. Then, for every $R > 0$, $t_0 \in I$, $x_0 \in \Omega$ with $I_{4R}(t_0) \times B_{4R}(x_0) \Subset I \times \Omega$, and any solution $u$ to $\partial_t u - L_t u = 0$ in $I \times \Omega$ satisfying $u \ge 0$ in $I \times \Omega$ it holds
	\begin{align}
	\begin{split}
		\label{eq:fPHI-minustail}
		&\dashint_{I^{\ominus}_{R}(t_0 - R^{\alpha})} \tail(u_+(t);R,x_0) \d t \; + \sup_{I^{\ominus}_{R}(t_0 - R^{\alpha}) \times B_R(x_0)} u \\ 
		& \qquad\qquad\qquad \le c \inf_{I_{R}^{\oplus}(t_0) \times B_R(x_0)} u + c \;  \dashint_{I_{4R}(t_0)} \tail(u_-(t);4R,x_0) \d t.
		\end{split}
	\end{align}
	Here, $c = c(d,\alpha_0,\lambda,\Lambda) > 0$ is a constant.
\end{corollary}

In \autoref{sec:ext} we treat jumping kernels that do not satisfy the pointwise bounds \eqref{eq:Kcomp}. \autoref{thm:fHI_sharp-ass} provides the full parabolic Harnack inequality in a general setting. In \autoref{thm:fHI_singular} we establish a Harnack-type inequality for operators with a jumping measure that is not absolutely continuous with respect to the Lebesgue measure. Such a result seems to be new in the literature. 

Our proof of the full Harnack inequality \autoref{thm:fHI} is based on a local boundedness estimate for subsolutions (see \autoref{thm:locbdL1tail}) and a weak Harnack inequality for supersolutions (see \autoref{thm:improvedwHI-parabolic}), both involving suitable tail terms that are $L^1$ in time (see \eqref{eq:tail-def}, \eqref{eq:L1-tail-def}). When explaining the strategy of the proof below we comment on how both of these auxiliary results improve the existing results in the literature. As an application of \autoref{thm:locbdL1tail} and \autoref{thm:improvedwHI-parabolic}, we are able to establish parabolic Hölder regularity estimates for weak solutions $u$ to $\partial_t u - L_t u = 0$. 

Somewhat surprisingly, it turns out that the parabolic H\"older estimate fails to hold in general if one works under the same assumptions on the weak solution $u$ as in \autoref{thm:fHI}. One needs to impose additional assumptions on the tail term. Note that this phenomenon does not exist in the elliptic case, where no additional assumptions are required. We construct a counterexample (see \autoref{ex:counterexample}), stating that the parabolic H\"older estimate fails if the tail is only $L^1$ in time. 

\begin{theorem}[parabolic H\"older regularity]
	\label{thm:HRE}
	Assume that $K$ satisfies \eqref{eq:Kcomp}. Then, for every $R > 0$, $t_0 \in I$, $x_0 \in \Omega$ with $I_{4R}^{\ominus}(t_0) \times B_{4R}(x_0) \subset I \times \Omega$, and any solution $u$ to $\partial_t u - L_t u = 0$ in $I \times \Omega$ satisfying $\tail(u;R/2,x_0) \in L^{1+\eps}_{loc}(I)$ for some $\eps > 0$, it holds $u \in C^{0,\gamma}(I_{R}^{\ominus}(t_0) \times B_{R}(x_0))$ for some $\gamma = \gamma(d,\alpha_0,\lambda,\Lambda) \in (0,1)$, and
	\begin{align*}
		&R^{\gamma} \frac{|u(t,x) - u(s,y)|}{\left(|t-s|^{1/\alpha} + |x-y|\right)^{\gamma}} \\ &\quad \le c \left( \dashint_{I_{2R}^{\ominus}(t_0) \times B_{2R}(x_0)} \hspace{-0.8cm} u^2(t,x) \d x \d t \right)^{1/2} \hspace{-0.2cm} + c \left(\dashint_{I^{\ominus}_{2R}(t_0)} \hspace{-0.4cm}\tail(u(t);R/2,x_0)^{1+\eps} \d t \right)^{\frac{1}{1+\eps}}
	\end{align*}
	for every $(t,x), (s,y) \in I^{\ominus}_{R}(t_0) \times B_{R}(x_0)$.
	Here, $c = c(d,\alpha_0,\lambda,\Lambda) > 0$.
\end{theorem}

\medskip

\begin{remark}
	\begin{itemize}
		\item[(i)] To the best of our knowledge, even in the case $L_t = (-\Delta)^{\alpha/2}$, Hölder estimates have not been proved under the assumption on the tail of $u$  above.
		\item[(ii)] As the counterexample shows (see \autoref{ex:counterexample}), the condition on the tail in \autoref{thm:HRE} is sharp in the range of $L^p$-spaces. We believe that \autoref{thm:HRE} can be further refined to yield a continuity estimate for weak solutions whose tail is in a suitable Lorentz space contained in $L^{1}_{loc}(I)$.
		\item[(iii)] If $u$ belongs to the energy space $L^2_{loc}(I ; V^{\alpha/2}(\Omega | \R^n)) \cap L^{\infty}_{loc}(I;L^2_{loc}(\Omega))$, then the condition $\tail(u;R/2,x_0) \in L^{1+\eps}_{loc}(I)$ is automatically satisfied for every $\eps \leq 1$  (see \autoref{lemma:tail-finite}). See \autoref{sec:prelim} for the definition of the function spaces under consideration.

	\end{itemize}
\end{remark}

\begin{remark}
\label{rem:robust}
	An important feature of our approach is that the main results \autoref{thm:fHI} and \autoref{thm:HRE}, are robust in the following sense: All constants in the assertions remain bounded as $\alpha \nearrow 2$. Thus, our results extend the De Giorgi-Nash-Moser theory for local equations of the form
	\begin{align*}
		\partial_t u(t,x) - \dvg(A(t,x) \nabla u(t,x)) = 0.
	\end{align*}
in a continuous way. For a more detailed discussion, we refer to \cite{KaWe22a} and \cite{CKW23}.
\end{remark}	
		
\subsection*{Related results} 

The regularity theory for elliptic and parabolic equations governed by nonlocal operators modeled upon the fractional Laplacian has become a highly active field of research in the past 20 years. 

In particular, the De Giorgi-Nash-Moser theory has successfully been extended to integro-differential operators satisfying \eqref{eq:Kcomp}: Starting with \cite{Kas09}, \cite{DyKa20}, and \cite{DKP14}, \cite{DKP16}, \cite{Coz17}, H\"older regularity estimates and Harnack inequalities have been established for weak solutions to the corresponding elliptic equation by nonlocal versions of the Moser iteration and the De Giorgi iteration, respectively.  Moreover, the aforementioned results have been extended to more general classes of operators including operators with singular jumping measures (see \cite{KaSc14}, \cite{ChKa20}, \cite{CKW23}), nonlinear operators with nonstandard growth (see \cite{ChKi23}, \cite{BOS22}, \cite{CKW22a}, \cite{CKW22b}, \cite{BKO22}, \cite{Ok23}), and nonsymmetric operators (see \cite{KaWe22a}, \cite{KaWe22b}). For parabolic equations, weak Harnack inequalities, H\"older regularity estimates, and local boundedness estimates have been investigated in \cite{CCV11}, \cite{FeKa13}, \cite{KaWe22a}, \cite{KaWe22b}, \cite{Str19a}, \cite{Str19b}, \cite{Kim20}, \cite{Lia24}, \cite{DZZ21}, \cite{APT22}. 

A parabolic Harnack inequality for nonlocal operators under \eqref{eq:Kcomp} has been established for the first time in \cite{BaLe02}, \cite{ChKu03}. In both articles the authors make substantial use of the associated stochastic jump process. In the significant contributions \cite{CKW20}, \cite{CKW21} the authors have established parabolic Harnack inequalities for nonlocal Dirichlet forms on doubling metric measure spaces. In particular, they are able to characterize these inequalities in terms of equivalent conditions on the geometry of the underlying space and suitable assumptions on the jumping kernel. As in \cite{BaLe02}, \cite{ChKu03} the approach is based on the stochastic process and uses heat kernel estimates in order to deduce the parabolic Harnack inequality. Note that \cite{CKW20}, \cite{CKW21} work under assumptions much weaker than \eqref{eq:Kcomp}, which we assume for clarity of the presentation (see \autoref{rem:discussion-full-harnack} (vi)). For related results on heat kernel estimates on metric measure spaces, see also \cite{GHKS14}, \cite{GHKS15}, \cite{GHH18}. Pointwise estimates are important both theoretically and from the point of view of applications, e.g. in results such as in \cite{BFV18}.

Finally, we point out that for elliptic and parabolic equations governed by operators in non-variational form,  H\"older regularity estimates and the Harnack inequality are proved in several works including \cite{BaKa05}, \cite{Sil06}, \cite{CaSi09}, \cite{Sil11}, \cite{KiLe12}, \cite{ChD12}, \cite{KiLe13a}, \cite{KiLe13b}, \cite{ChD14}, \cite{ChDa16}, \cite{ScSi16}, \cite{DRSV22}.

\subsection*{Strategy of proof}

As in the classical proof by Moser (see \cite{Mos64}, \cite{Mos71}), the Harnack inequality for nonlocal equations is a combination of the weak Harnack inequality for supersolutions and the local boundedness estimate for subsolutions. A major difference between nonlocal and local equations arises from the need to take into account long range interactions, which enter the respective estimates through tail terms defined as
\begin{align*}
\tail(v;R,x_0) = (2-\alpha) \int_{\R^d \setminus B_R(x_0)} |v(y)||x_0 - y|^{-d-\alpha} \d y.
\end{align*}
For instance, in \cite{DKP16} and \cite{Coz17}, it was proved that subsolutions to the elliptic equation $-L u = 0$ in $B_{2R}(x_0)$ satisfy the following local boundedness estimate:
\begin{align}
\label{eq:intro_lb}
\sup_{B_{R/2}(x_0)} u_+ \le c(\delta)\left(\dashint_{B_{R}(x_0)} u_+(x) \d x \right) + \delta \tail(u_+;R,x_0).
\end{align}
In order to establish a full Harnack inequality for globally nonnegative solutions, they made use of the interpolation between the local and nonlocal contributions which is kept in the parameter $\delta > 0$ in \eqref{eq:intro_lb} and of the fact that
\begin{align}
\label{eq:intro_tail-est}
\tail(u_+;R,x_0) \le c \sup_{B_{3R/2}(x_0)} u
\end{align}
for nonnegative supersolutions. In fact, by a covering argument and choosing $\delta > 0$ small enough, \eqref{eq:intro_tail-est} makes it possible to absorb the tail term in \eqref{eq:intro_lb} on the left hand side and thereby to prove the full Harnack inequality without any tail terms after combination with the weak Harnack inequality (see \cite{Kas09}, \cite{DKP14}, \cite{Coz17}, \cite{DyKa20}), which reads as
\begin{align*}
\left(\dashint_{B_{R}(x_0)} u(x) \d x \right) \le c \inf_{B_{R/2}(x_0)} u.
\end{align*}

In the parabolic case, the current state of the art are the following parabolic versions of the local boundedness estimate for subsolutions and the weak Harnack inequality for nonnegative supersolutions (e.g. in \cite{FeKa13}, \cite{Kim19}, \cite{Str19b}, \cite{Kim20}, \cite{DZZ21}, \cite{BGK21}, \cite{GaKi21}, \cite{Nak22}, \cite{FSZ22}, \cite{Lia24}, \cite{APT22}, \cite{KaWe22a}, \cite{KaWe22b}, \cite{KaWe23}) to the equation $\partial_t u - L_t u = 0$:
\begin{align}
\sup_{I^{\ominus}_{R/2}(t_0) \times B_{R/2}(x_0)} u_+  &\le c(\delta) \left(\dashint_{I^{\ominus}_R(t_0) \times B_R(x_0)} u_+(t,x) \d x \d t \right) \nonumber \\
& \quad + \delta \sup_{t \in I^{\ominus}_R(t_0)} \tail(u_+(t);R/2,x_0), \label{eq:intro_parlb} \\
\label{eq:intro_parwH}
\inf_{I_{R}^{\oplus}(t_0) \times B_{R}(x_0)} u &\ge c\dashint_{I_{R}^{\ominus}(t_0 - R^{\alpha}) \times B_{R}(x_0)} u(t,x) \d x \d t.
\end{align} 

These two estimates cannot be combined in order to deduce a full Harnack inequality. This is due to the appearance of the so-called $L^\infty$-tail in \eqref{eq:intro_parlb}
\begin{align}\label{eq:Linfty-tail-def}
\sup_{t \in I^{\ominus}_R(t_0)} \tail(u_+(t);R,x_0).
\end{align}

Although the $L^\infty$-tail seems to appear automatically when naively adapting the nonlocal De Giorgi and Moser iteration to the parabolic case, it is not naturally associated to the parabolic equation. In fact, the $L^\infty$-tail is not necessarily finite for weak solutions in the natural energy space $u \in L^2_{loc}(I ; V^{\alpha/2}(\Omega | \R^d))$. See also the discussion in \autoref{sec:prelim}.  Moreover, (or rather as a consequence) it does not seem to be controllable by a local quantity such as in the elliptic case (see \eqref{eq:intro_tail-est}). 

The main contribution of this article is to improve both, the local boundedness estimate for subsolutions \eqref{eq:intro_parlb} and the weak Harnack inequality for supersolutions \eqref{eq:intro_parwH} in such a way that they imply the full Harnack inequality \autoref{thm:fHI}. A central object in this article is the $L^1$-tail 
\begin{align}
\label{eq:L1-tail-def}
\int_{I^{\ominus}_R(t_0)} \tail(u(t);R,x_0) \d t,
\end{align}
which is a natural quantity as opposed to the $L^\infty$-tail because it is finite for any function in the natural weak solution space (see \autoref{def:col-concept}). Moreover, it can be estimated by the supremum of the solution in a ball in analogy to \eqref{eq:intro_tail-est}. 

\subsubsection*{Local boundedness} First, we establish an improved version of the local boundedness estimate for subsolutions \eqref{eq:intro_lb}, which involves the $L^1$-tail instead of the $L^\infty$-tail:

\begin{theorem}[local boundedness with $L^1$-tail]
\label{thm:locbdL1tail}
Assume that $K$ satisfies \eqref{eq:Kcomp}. Then, for every $R > 0$, $t_0 \in I$, $x_0 \in \Omega$ with $I_{4R}(t_0) \times B_{4R}(x_0) \subset I \times \Omega$ and any subsolution $u$ to $\partial_t u - L_t u = 0$ in $I \times \Omega$
\begin{align}
\label{eq:locbdL1tail}
\begin{split}
\sup_{I^{\ominus}_{R/2}(t_0) \times B_{R/2}(x_0)} u_+ &\le c \left(\dashint_{I^{\ominus}_R(t_0) \times B_R(x_0)} u_+^2(t,x) \d x \d t \right)^{1/2} \\
&\quad + c\dashint_{I^{\ominus}_R(t_0)} \tail(u_+(t);R,x_0) \d t.
\end{split}
\end{align}
Here, $c = c(d,\alpha_0,\lambda,\Lambda) > 0$ is a constant. In particular, since all terms on the right hand side are finite, $u_+ \in L^{\infty}(I^{\ominus}_{R/2}(t_0) \times B_{R/2}(x_0))$.
\end{theorem}

An immediate consequence of \autoref{thm:locbdL1tail} is that weak solutions  $u$ to $\partial_t u - L_t u = 0$ are locally bounded. This property has not been established yet without further assumptions.

So far, for nonlocal operators in the framework of this article, \eqref{eq:locbdL1tail} has been proved in \cite{Str19b} for global solutions. The technique therein cannot be extended to allow for local solutions. However, for nonlocal operators in non-variational form, \eqref{eq:locbdL1tail} can be found for instance in \cite{ChDa16}. 

Our proof of \autoref{thm:locbdL1tail} is based on a decomposition of the tail into two parts. The first part takes care of the values of the solution inside the solution domain and can be treated by using the solution property itself. The second part contains information about $u$ away from the singularity of the kernel and can therefore be controlled much easier. For more details, we refer the reader to \autoref{sec:lb_first-proof} and \autoref{sec:lb_second-proof}, where we give two different proofs of \autoref{thm:locbdL1tail}.

An important observation when comparing \eqref{eq:locbdL1tail} to its elliptic counterpart \eqref{eq:intro_lb} or to \eqref{eq:intro_parlb} is that \eqref{eq:locbdL1tail} does not contain an interpolation parameter in front of the tail term. In fact, an interpolation between the $L^1$-tail and the average of $u$ would be unnatural since the $L^1$-tail is critical with respect to scaling. This is another difficulty of the parabolic case when compared to the elliptic case. As a consequence, \autoref{thm:locbdL1tail} is not good enough to imply the full Harnack inequality without an improvement of the weak Harnack inequality \eqref{eq:intro_parwH}.

\subsubsection*{Weak Harnack inequality} As explained above, the weak Harnack inequality \eqref{eq:intro_parwH} has to be improved in order to be combined with \autoref{thm:locbdL1tail}. We establish a new weak Harnack inequality for supersolutions, which contains an upper bound for the $L^1$-tail of $u$ by its infimum in a ball.

\begin{theorem}[weak Harnack inequality with $L^1$-tail]
\label{thm:improvedwHI-parabolic}
Assume that $K$ satisfies \eqref{eq:Kcomp}. Then, for every $R > 0$, $t_0 \in I$, $x_0 \in \Omega$ with $I_{4R}(t_0) \times B_{4R}(x_0) \subset I \times \Omega$, and any globally nonnegative supersolution $u$ to $\partial_t u-L_t u=0$ in $I \times \Omega$ it holds
\begin{align}
\label{eq:improvedwHI}
\dashint_{I^{\ominus}_{R}(t_0 - R^{\alpha}) \times B_R(x_0)
} \hspace{-0.8cm} u(t,x) \d x \d t + \dashint_{I^{\ominus}_{R}(t_0 - R^{\alpha})} \hspace{-0.6cm} \tail(u(t);R,x_0) \d t \le c \inf_{I_{R}^{\oplus}(t_0) \times B_{R}(x_0)} u. 
\end{align}
Here, $c = c(d,\alpha_0,\lambda,\Lambda) > 0$ is a constant.
\end{theorem}

Note that \autoref{thm:improvedwHI-parabolic} is already new in the elliptic case for operators in the framework of this article. So far, the best known estimate for the tail of a supersolution is given in \eqref{eq:intro_tail-est}. However, \autoref{thm:improvedwHI-parabolic} implies that the tail of a supersolution $u$ can even by bounded by its infimum in a ball. For the proof of \autoref{thm:improvedwHI-parabolic}, we refer to \autoref{sec:whi}. For operators in non-variational form an estimate similar to \eqref{eq:improvedwHI} can be found for instance in \cite{ChDa16}. 

Having at hand the local boundedness estimate with $L^1$-tail and the weak Harnack inequality with $L^1$-tail, our main results (see \autoref{thm:fHI}, \autoref{cor:fPHI}, and \autoref{thm:HRE}) follow by standard arguments. For more details on the proofs, we refer to \autoref{sec:main-proofs}.

\begin{remark}[Next steps]
It is an interesting task to extend our approach. It would be desirable to cover certain nonlinear equations involving, e.g.
\begin{itemize}
\item the heat equation governed by the fractional $p$-Laplacian, thus improving \cite{DZZ21}, \cite{APT22}, \cite{Lia24} making use of Di Benedetto's intrinsic scaling method,
\item the doubly nonlinear nonlocal parabolic equation (Trudinger's equation), cf. \cite{BGK21},
\item mixed local-nonlocal parabolic $p$-Laplace equations, see also \cite{GaKi21}, \cite{GaKi23}, \cite{FSZ22}, \cite{Nak22}, \cite{Nak23}, \cite{ShZh23}.
\end{itemize}

Moreover, it is challenging to investigate a possible generalization of our approach to kinetic equations, such as the fractional Kolmogorov equation (see \cite{ImSi19}, \cite{Loh24}). \\
Last, let us mention an interesting detail. For partial differential operators of second order, due to the shape of the fundamental solution, the time-lag in the parabolic Harnack inequality is known to be inevitable. This phenomenon is different for nonlocal operators of fractional order (see \cite{BSV17}, \cite{DKSZ20}). It would be interesting to study the parabolic Harnack inequality without the time-lag between $I^{\ominus}_{R}(t_0 - R^{\alpha})$ and $I_{R}^{\oplus}(t_0)$.
\end{remark}

\textbf{Acknowledgement:} The authors thank Luis Silvestre for helpful discussions concerning \autoref{lemma:locbdsuptail} in the case $\mu=1, \theta = \infty$.

\subsection*{Outline}
This article consists of six sections. In \autoref{sec:prelim}, we introduce the weak solution concept and list several auxiliary results. In \autoref{sec:lb} and \autoref{sec:whi}, we establish the local boundedness estimate with $L^1$-tail (see \autoref{thm:locbdL1tail}) and the weak Harnack inequality with $L^1$-tail (see \autoref{thm:improvedwHI-parabolic}). Our main results (see \autoref{thm:fHI}, \autoref{cor:fPHI}, and \autoref{thm:HRE}) are proved in \autoref{sec:main-proofs}. Finally, in \autoref{sec:ext}, we discuss possible extensions of our main results to a more general class of jumping kernels. 

\section{Preliminaries}
\label{sec:prelim}

In this section we introduce the weak solution concept and suitable function spaces. Moreover, we list some auxiliary results.

First, we introduce the following function spaces:
Given an open set $\Omega \subset \R^d$, we define:
\begin{align*}
V^{\alpha/2}(\Omega|\R^d) &= \Big\lbrace v : v \hspace{-0.1cm}\mid_{\Omega} \in L^2(\Omega),\\
& \qquad \qquad [v]_{V^{\alpha/2}(\Omega|\R^d)}^2 := (2-\alpha) \int_{\Omega}\int_{\R^d}\frac{(v(x)-v(y))^2}{|x-y|^{d+\alpha}} \d y \d x < \infty \Big\rbrace,\\
H^{\alpha/2}_{\Omega}(\R^d) &= \left\lbrace v \in V^{\alpha/2}(\R^d|\R^d) : \supp(v) \subset \Omega \right\rbrace,\\
H^{\alpha/2}(\Omega) &= \Big\lbrace v \in L^2(\Omega) : [v]^2_{H^{\alpha/2}(\Omega)} := (2-\alpha) \int_{\Omega}\int_{\Omega} \frac{(v(x)-v(y))^2}{|x-y|^{d+\alpha}} \d y \d x < \infty \Big\rbrace,\\
L^1_{\alpha}(\R^d) &= \Big\lbrace v : \Vert v \Vert_{L^1_{\alpha}(\R^d)} := \int_{\R^d} |v(x)| (1+ |x|)^{-d-\alpha} \d x < \infty \Big\rbrace,
\end{align*}
equipped with
\begin{align*}
\Vert v \Vert_{V^{\alpha/2}(\Omega|\R^d)}^2 &= \Vert v \Vert_{L^2(\Omega)}^2 + [v]_{V^{\alpha/2}(\Omega|\R^d)}^2,\\
\Vert v \Vert_{H^{\alpha/2}_{\Omega}(\R^d)}^2 &= \Vert v \Vert_{L^2(\R^d)}^2 + [v]_{V^{\alpha/2}(\R^d|\R^d)}^2,\\
\Vert v \Vert_{H^{\alpha/2}(\Omega)}^2 &= \Vert v \Vert_{L^2(\Omega)}^2 + [v]^2_{H^{\alpha/2}(\Omega)}.
\end{align*}

Note that $V^{\alpha/2}(\R^d | \R^d) = H^{\alpha/2}(\R^d)$. We will use both terms synonymously.

Moreover, given a jumping kernel $K$, we associate it with the bilinear forms
\begin{align*}
\cE^{K(t)}(u,v) = \int_{\R^d}\int_{\R^d} (u(x) - u(y))(v(x) - v(y)) K(t;x,y) \d y \d x, ~~ t > 0.
\end{align*}
Note that $(L_t u , v)_{L^2(\R^d)} = \cE^{K(t)}(u,v)$. Given a set $M \subset \R^d \times \R^d$, we introduce the notation
\begin{align*}
\cE^{K(t)}_M(u,v) = \iint_{M} (u(x) - u(y))(v(x) - v(y)) K(t;x,y) \d y \d x.
\end{align*}
If $B \subset \R^d$, we also write $\cE^{K(t)}_B(u,v) := \cE^{K(t)}_{B \times B}(u,v)$.

The space $V^{\alpha/2}(\Omega | \R^d)$ can be regarded as the natural space to analyze weak solutions to the elliptic problem $-L u = 0$ in $\Omega$. In fact, this PDE is the Euler-Lagrange equation of the minimization problem 
\begin{align*}
u \mapsto \cE^{K}_{(\R^d \times \R^d) \setminus (\Omega^c \times \Omega^c)}(u,u),
\end{align*}
and clearly, due to \eqref{eq:Kcomp}, the energy is finite if and only if $u \in V^{\alpha/2}(\Omega | \R^d)$. Therefore, the natural energy space to analyze weak solutions to the parabolic problem $\partial_t u - L_t u = 0$ in $I \times \Omega$ is $L^2_{loc}(I;V^{\alpha/2}(\Omega|\R^d)) \cap L^{\infty}_{loc}(I;L^2_{loc}(\Omega))$.\\ However, it turns out that all quantities in the corresponding weak formulation (see \eqref{eq:sol-concept}) are finite even under weaker assumptions. In particular, we have that $\cE^{K}(u,\phi) < \infty$ for $\phi \in H^{\alpha/2}_{\Omega}(\R^d)$ and $u \in H_{loc}^{\alpha/2}(\Omega) \cap L^1_{\alpha}(\R^d)$. Therefore, the tail terms do not have to be $L^2$ in space.

This discussion leads us to the following definition of the weak solution concept:

\begin{definition}[weak solution concept]
\label{def:col-concept}
Let $f \in L^{1}_{loc}(I; L^1_{loc}(\Omega))$.\\
We say that $u \in L^2_{loc}(I;H_{loc}^{\alpha/2}(\Omega)) \cap L^1_{loc}(I; L^1_{\alpha}(\R^d)) \cap L^{\infty}_{loc}(I;L^2_{loc}(\Omega))$ is a weak supersolution to $\partial_t u - L_t u = f$ in $I \times \Omega$ if the weak $L^2(\Omega)$-derivative $\partial_t u$ exists, $\partial_t u \in L^1_{loc}(I;L^2_{loc}(\Omega))$ and 
\begin{equation}
\label{eq:sol-concept}
(\partial_t u(t),\phi) + \cE^{K(t)}(u(t),\phi) \le (f(t),\phi), \quad \forall t \in I,~ \forall \phi \in H_{\Omega}^{\alpha/2}(\R^d) \text{ with } \phi \le 0.
\end{equation}
$u$ is called a weak subsolution if \eqref{eq:sol-concept} holds true for every $\phi \ge 0$. $u$ is called a weak solution, if it is a supersolution and a subsolution.
\end{definition}

Note that the regularity assumptions for weak solutions are natural in the sense that they are the nonlocal analogs of the assumptions in the famous articles by Moser (see \cite{Mos64}, \cite{Mos71}).

Clearly, the $L^1$-tail is finite for functions in $ L^1_{loc}(I; L^1_{\alpha}(\R^d))$. Therefore, it is finite for weak solutions in the natural function space according to \autoref{def:col-concept}. The following lemma states that if $u$ is assumed to belong to the slightly smaller energy space, also the $L^2$-tail (as opposed to the $L^\infty$-tail) is still finite.

\begin{lemma}
\label{lemma:tail-finite}
Let $u \in L^2_{loc}(I;V^{\alpha/2}(\Omega|\R^d))$. Then, for any $R > 0$ and $(t_0,x_0) \in \Omega$ such that $B_{2R}(x_0) \subset \Omega$ and $I^{\ominus}_R(t_0) \subset I$, it holds for some constant $c = c(n,\alpha_0,R) > 0$
\begin{align*}
\int_{I_{R}^{\ominus}(t_0)} \tail(u(t);R,x_0)^2 \d t \le c\int_{I_{R}^{\ominus}(t_0)} \Vert u(t) \Vert_{V^{\alpha/2}(B_{2R}(x_0) | \R^d)}^2 \d t < \infty.
\end{align*}
\end{lemma}

\begin{proof}
See Proposition 13 in \cite{DyKa19} or Lemma 2.13 and Remark 2.14 in \cite{KaWe22b}. 
\end{proof}

The following classical iteration lemma will be used several times throughout our proof.

\begin{lemma}
\label{lemma:it}
Let $R > 0$ and $f : [R/2, R] \to [0,\infty)$ be bounded. Assume that there are $A, B, C, \gamma_1, \gamma_2 > 0$ and $\theta \in (0,1)$ such that for any $R/2 \le r \le s \le R$ it holds
\begin{align*}
f(r) \le A(s-r)^{-\gamma_1} + B(s-r)^{-\gamma_2} + C + \theta f(s).
\end{align*}
Then, there exists a constant $c > 0$, depending only on $\theta, \gamma_1, \gamma_2$, such that
\begin{align*}
f(R/2) \le c R^{-\gamma_1} A + c R^{-\gamma_2} B + cC.
\end{align*}
\end{lemma}

\begin{proof}
The proof is a direct consequence of \cite[Lemma 1.1]{GiGi82}.
\end{proof}

\section{Proof of local boundedness with $L^1$-tail}
\label{sec:lb}

The goal of this section is to establish the improved local boundedness estimate for weak subsolutions (see \autoref{thm:locbdL1tail}). In comparison to the existing results in the literature, we are able to control the supremum of a subsolution in a ball by a nonlocal quantity that is only $L^1$ in time, while all previous results contained the $L^\infty$-tail.\\ 
Our proof is based on a decomposition of the $L^\infty$-tail into two parts, which is carried out by a cutoff procedure on the solution. While the first part (intermediate tail) takes care of the values of the solution inside the solution domain, the second part (remaining tail) only contains the information outside the solution domain. Both terms will be treated separately. For the intermediate tail, we are able to make use of the equation, which allows us to estimate it by an $L^1$-tail (see \autoref{lemma:intermediate-tail}). By treating the remaining tail like a lower order source term, we are able to improve its integrability in the local boundedness estimate. We propose two different approaches to treat the remaining tail. While the first approach is based on a localization argument for the subsolution (see \autoref{sec:lb_first-proof}), the second approach goes via a modification of the level set truncation in the De Giorgi iteration (see \autoref{sec:lb_second-proof}).

Let us first provide the main ingredient for the treatment of the intermediate tail. We believe this result to be of independent interest.

\begin{lemma}
\label{lemma:intermediate-tail}
Let $u$ be a subsolution to $\partial_t u - L_t u = 0$ in $I \times \Omega$. Let $L_t$ be as before. Then, for every $R > 0$, $t_0 \in I$, $x_0 \in \Omega$ with $I_{4R}(t_0) \times B_{4R}(x_0) \subset I \times \Omega$, it holds
\begin{align*}
\sup_{I^{\ominus}_{R/2}(t_0)} \left( \dashint_{B_{R/2}(x_0)} \hspace{-0.2cm} u_+^2(t,x) \d x \right)^{1/2} &\le c \left(\dashint_{I^{\ominus}_R(t_0) \times B_R(x_0)} \hspace{-0.2cm} u_+^2(t,x) \d x \d t \right)^{1/2} \\
&\quad + c \dashint_{I^{\ominus}_R(t_0)} \hspace{-0.2cm} \tail(u_+(t);R,x_0) \d t.
\end{align*}
Here, $c = c(d,\alpha_0,\lambda,\Lambda) > 0$ is a constant.
\end{lemma}

\begin{proof}
Let $R/2 \le r \le s \le R$. By the Caccioppoli inequality, using \eqref{eq:Kcomp}, we obtain
\begin{align*}
&\sup_{t \in I_r^{\ominus}(t_0)} \dashint_{B_r(x_0)} u_+^2(t,x) \d x \le c (s-r)^{-\alpha} \int_{I_s^{\ominus}(t_0)} \dashint_{B_s(x_0)} u_+^2(t,x) \d x \d t\\
&\quad+ c \left( \frac{s-r}{R} \right)^{-d-\alpha} \hspace{-0.2cm} \int_{I_{\frac{r+s}{2}}^{\ominus}(t_0)} \left( \dashint_{B_{\frac{r+s}{2}}(x_0)} u_+(t,x) \d x \right) \\
&\qquad\qquad\qquad\qquad\qquad\qquad ~~ \times\left( \int_{\R^d \setminus B_{s}(x_0)} \hspace{-0.2cm} u_+(t,y) \frac{(2-\alpha)}{|y - x_0|^{d+\alpha}} \d y \right) \d t.
\end{align*}
Note that the aforementioned Caccioppoli inequality is by now classical and can be found e.g. in (3.12) in \cite{KaWe22b} (see also Lemma 2.4 in \cite{KaWe23} or \cite{Str19b}). It can be established by testing the weak formulation \eqref{eq:sol-concept} with $\phi = \tau^2 u_+$, where $\tau$ is a suitable cutoff function between $B_r(x_0)$ and $B_{\frac{r+s}{2}}(x_0)$ with $|\nabla \tau| \le c (s-r)^{-1}$ before multiplying both sides with a cutoff function $\chi$ between $I_r^{\ominus}(t_0)$ and $I_{\frac{r+s}{2}}^{\ominus}(t_0)$ with $|\chi'| \le c((r+s)^{\alpha} - r^{\alpha})^{-1}$.
For the second term, we estimate
\begin{align*}
&c\left( \frac{s-r}{R} \right)^{-d-\alpha}\int_{I_{\frac{r+s}{2}}^{\ominus}(t_0)} \left(\dashint_{B_{\frac{r+s}{2}}(x_0)} u_+(t,x) \d x \right) \\
&\qquad\qquad\qquad\qquad\qquad ~~ \times \left( \int_{\R^d \setminus B_{s}(x_0)} u_+(t,y) \frac{(2-\alpha)}{|y - x_0|^{d+\alpha}} \d y \right) \d t\\
&\le \frac{1}{2} \sup_{t \in I_s^{\ominus}(t_0)} \left(\dashint_{B_s(x_0)}  \hspace{-0.2cm} u_+(t,x) \d x \right)^2\\
&\quad + c  \left(\left( \frac{s-r}{R} \right)^{-d-\alpha} \hspace{-0.2cm} \int_{I_s^{\ominus}(t_0)} \int_{\R^d \setminus B_{s}(x_0)} \hspace{-0.2cm} u_+(t,y) \frac{(2-\alpha)}{|y - x_0|^{d+\alpha}} \d y \d t \right)^2\\
&\le  \frac{1}{2} \sup_{t \in I_s^{\ominus}(t_0)} \dashint_{B_s(x_0)}  \hspace{-0.2cm} u_+^2(t,x) \d x \\
&\quad + c \left( \frac{s-r}{R} \right)^{-2(d+\alpha)} \hspace{-0.2cm} \left( \int_{I_R^{\ominus}(t_0)} \int_{\R^d \setminus B_{R/2}(x_0)} \hspace{-0.2cm} u_+(t,y) \frac{(2-\alpha)}{|y - x_0|^{d+\alpha}} \d y \d t \right)^2.
\end{align*}
Setting
\begin{align*}
A &= c\int_{I_R^{\ominus}(t_0)} \dashint_{B_R(x_0)} u_+^2(t,x) \d x \d t,\\
B &= c R^{2(d+\alpha)} \left( \int_{I_R^{\ominus}(t_0)} \int_{\R^d \setminus B_{R/2}(x_0)} u_+(t,y) \frac{(2-\alpha)}{|y - x_0|^{d+\alpha}} \d y \d t \right)^2,\\
f(r) &= \sup_{t \in I_r^{\ominus}(t_0)} \dashint_{B_r(x_0)} u_+^2(t,x) \d x,
\end{align*}
we have shown that
\begin{align*}
f(r) \le (s-r)^{-\alpha} A + (s-r)^{-2(d+\alpha)} B + \frac{1}{2} f(s), ~~ \forall R/2 \le r \le s \le R.
\end{align*}
Note that $f$ is bounded since $u \in L^{\infty}_{loc}(I;L_{loc}^2(\Omega))$ by assumption. Therefore, a classical iteration lemma (see \autoref{lemma:it}) gives
\begin{align*}
f(R/2) \le c R^{-\alpha} A + c R^{-2(d+\alpha)} B,
\end{align*}
which yields our desired result.
\end{proof}

\subsection{Proof via localization argument}
\label{sec:lb_first-proof}

In this section we give a first proof of \autoref{thm:locbdL1tail} which relies on the existing local boundedness estimate involving the $L^\infty$-tail and a localization argument through which we will split the tail into two parts. 
In preparation for the proof, we first state the local boundedness estimate involving the $L^\infty$-tail:

\begin{lemma}
\label{lemma:locbdsuptail}
Let $u$ be a subsolution to $\partial_t u - L_t u = f$ in $I \times \Omega$, where $f \in L^{\mu,\theta}_{t,x}(I^{\ominus}_{R}(t_0) \times B_R(x_0))$ with
\begin{align*}
\frac{1}{\mu} + \frac{d}{\alpha \theta} < 1 ~~ \text{ or } (\mu,\theta) = (1,\infty).
\end{align*}
Let $L_t$ be as before. Then, for every $\delta >0$, $R > 0$, $t_0 \in I$, $x_0 \in \Omega$ with $I_{4R}(t_0) \times B_{4R}(x_0) \subset I \times \Omega$:
\begin{align*}
\sup_{I^{\ominus}_{R/2}(t_0) \times B_{R/2}(x_0)} \hspace{-0.3cm} u_+ &\le c(\delta) \left(\dashint_{I^{\ominus}_R(t_0) \times B_R(x_0)} \hspace{-0.3cm} u_+^2(t,x) \d x \d t \right)^{1/2}\\
&\quad + \delta \sup_{I^{\ominus}_R(t_0)} \tail(u_+(t);R,x_0) + c R^{\alpha - \frac{d}{\theta} - \frac{\alpha}{\mu}} \Vert f_+ \Vert_{L^{\mu,\theta}_{t,x}(I^{\ominus}_{R}(t_0) \times B_R(x_0))}.
\end{align*}
Here, $c = c(d,\alpha_0,\lambda,\Lambda) > 0$ is a constant.
\end{lemma}

\begin{proof}
The proof in the case $f \in L^{\infty,\infty}_{t,x}$ is standard and can be found for instance in \cite{KaWe22b}, \cite{KaWe23}, \cite{Str19b}. Note that although a slightly smaller weak solution space is considered in these references, the proof directly carries over to our more general setting. For $f \in L^{\mu,\theta}_{t,x}$ with subcritical $(\mu,\theta)$, i.e. $\frac{1}{\mu} + \frac{d}{\alpha\theta} < 1$, one can use standard interpolation arguments as in the local case. Note that this case is not relevant for this article.\\
In case $f \in L^{1,\infty}_{t,x}$, we observe that since $u$ is a subsolution to $\partial_t u - L_t u=f$, it is a subsolution to $\partial_t u - L_t u=f_+$, and therefore the function
\begin{align*}
v(t,x) := u(t,x) - \int_{t_0}^t \Vert f_+(s) \Vert_{L^{\infty}(B_R(x_0))} \d s
\end{align*}
is a subsolution to $\partial_t v - L_t v = 0$. Therefore, the desired estimate follows by an application of \autoref{lemma:locbdsuptail} with $f = 0$ to $v$.
\end{proof}

When applying the localized solution to \autoref{lemma:locbdsuptail}, the $L^\infty$-tail will only see values of $u$ inside the solution domain (intermediate tail), which allows us to apply \autoref{lemma:intermediate-tail}. Moreover, the remaining tail will enter the estimate only as a source term and can therefore automatically be estimated by an $L^1$-norm in time.

We are now ready to give the first proof of \autoref{thm:locbdL1tail}.

\begin{proof}[Proof of \autoref{thm:locbdL1tail}]
Let us take $\eta \in C_c^{\infty}(B_{3R/2}(x_0))$ with $\eta \equiv 1$ in $B_{R}(x_0)$. We decompose
\begin{align*}
u = \eta u + (1-\eta) u
\end{align*}
and observe that $\eta u$ solves
\begin{align*}
\partial_t(\eta u) - L_t(\eta u) = L_t((1-\eta)u) \text{ in } I_{2R}(t_0) \times B_{2R}(x_0).
\end{align*}
In particular, for $t \in I_{2R}(t_0)$ and $x \in B_{R/2}(x_0)$ we compute using \eqref{eq:Kcomp}
\begin{align}
\label{eq:festimate}
\begin{split}
L_t((1-\eta) u)(x) &= \int_{\R^d \setminus B_{R}(x_0)} ((1-\eta)u(t,y)) K(t;x,y) \d y\\
&\le c (2-\alpha) \int_{\R^d \setminus B_{R}(x_0)} u_+(t,y) |y - x_0|^{-d-\alpha} \d y.
\end{split}
\end{align}
In particular, $f_+ := L_t((1-\eta) u)_+ \in L^{1,\infty}_{t,x}(I_{2R}(t_0) \times B_{R/2}(x_0))$. Therefore, by application of \autoref{lemma:locbdsuptail} to $\eta u$, we obtain:
\begin{align*}
\sup_{I^{\ominus}_{R/4}(t_0) \times B_{R/4}(x_0)} (\eta u)_+ &\le c \left(\dashint_{I^{\ominus}_{R/2}(t_0) \times B_{R/2}(x_0)} (\eta u)_+^2(t,x) \d x \d t \right)^{1/2}\\
&\quad + c (2-\alpha) R^{\alpha} \sup_{I^{\ominus}_{R/2}(t_0)} \int_{\R^d \setminus B_{R/4}(x_0)} \hspace{-0.2cm} (\eta u)_+(t,x) |x_0 - x|^{-d-\alpha} \d x\\
&\quad + c\Vert f_+ \Vert_{L^{1,\infty}_{t,x}(I^{\ominus}_{R/2}(t_0) \times B_{R/2}(x_0))}.
\end{align*}
By the properties of $\eta$, it follows:
\begin{align*}
& \sup_{I^{\ominus}_{R/4}(t_0) \times B_{R/4}(x_0)} u_+ \le c \left(\dashint_{I^{\ominus}_{R/2}(t_0) \times B_{R/2}(x_0)} u_+^2(t,x) \d x \d t \right)^{1/2}\\
&\qquad \qquad \qquad + c (2-\alpha) R^{\alpha} \sup_{I^{\ominus}_{R/2}(t_0)} \int_{B_{3R/2}(x_0) \setminus B_{R/4}(x_0)} \hspace{-0.2cm} u_+(t,x) |x_0 - x|^{-d-\alpha} \d x \\
&\qquad\qquad\qquad + c \Vert f_+ \Vert_{L^{1,\infty}_{t,x}(I^{\ominus}_{R/2}(t_0) \times B_{R/2}(x_0))}.
\end{align*}
Note that by \autoref{lemma:intermediate-tail}:
\begin{align*}
R^{\alpha} & \sup_{I^{\ominus}_{R/2}(t_0)} \int_{B_{3R/2}(x_0) \setminus B_{R/4}(x_0)} \hspace{-0.2cm} u_+(t,x) |x_0 - x|^{-d-\alpha} \d x \\
&\le c \sup_{I^{\ominus}_{R/2}(t_0)} \dashint_{B_{3R/2}(x_0)} u_+(t,x) \d x\\
&\le c \left(\dashint_{I^{\ominus}_{3R}(t_0) \times B_{3R}(x_0)} \hspace{-0.2cm} u_+^2(t,x) \d x \d t \right)^{1/2}\\
&\quad + c \dashint_{I^{\ominus}_{3R}(t_0)} \int_{\R^d \setminus B_{3R}(x_0)} \hspace{-0.2cm} u_+(t,x) |x_0 - x|^{-d-\alpha} \d x \d t.
\end{align*}
Therefore, using also \eqref{eq:festimate}, we obtain:
\begin{align*}
\sup_{I^{\ominus}_{R/4}(t_0) \times B_{R/4}(x_0)} u_+ &\le c \left(\dashint_{I^{\ominus}_{3R}(t_0) \times B_{3R}(x_0)} u_+^2(t,x) \d x \d t \right)^{1/2}\\
&+ c (2-\alpha) \dashint_{I^{\ominus}_{3R}(t_0)} \int_{\R^d \setminus B_{3R}(x_0)} u_+(t,x) |x_0 - x|^{-d-\alpha} \d x \d t.
\end{align*}
The aforementioned estimate can be proved for adjusted time-space cylinders, thus implying the desired estimate \eqref{eq:locbdL1tail}.
\end{proof}

\subsection{Proof via modified level set truncation in the De Giorgi iteration}
\label{sec:lb_second-proof}

We end this section with a second proof of \autoref{thm:locbdL1tail}. This proof makes use of \autoref{lemma:intermediate-tail} but not of the local boundedness estimate with $L^\infty$-tail (see \autoref{lemma:locbdsuptail}).

Recall the following observation from the proof of \autoref{lemma:locbdsuptail} in case $(\mu, \theta) = (1, \infty)$, namely that if $\partial_t u - L_t u \le f$, then also $\partial_t v - L_t v \le 0$, where $v = u - \int_{t_0}^t \Vert f_+(s) \Vert_{L^{\infty}} \d s$. The idea of the second proof of \autoref{thm:locbdL1tail} is to incorporate this observation into the De Giorgi iteration. This goes via a suitable modification of the level set truncation in the De Giorgi iteration. As an advantage of this approach, we do not have to localize the solution $u$ in order to consider the tail like a source term, but can directly profit from a cancellation of the tail terms within the iteration scheme. We believe that this idea might be useful since it could allow to generalize our technique to nonlinear operators, such as the fractional $p$-Laplacian.

\begin{proof}[Proof of \autoref{thm:locbdL1tail}]
For $l > 0$ we define
\begin{align*}
w_l(t,x) = \left(u(t,x) - C R^{-\alpha}\int_{t_0-(r+\rho)^{\alpha}}^t \tail(u_+(s);R,x_0) \d s - l \right)_+,
\end{align*}
where $C > 0$ is a constant that will be chosen suitably. Note that if $t \mapsto u(t,x)$ is differentiable
\begin{align}
\label{eq:derivativecomputation}
(\partial_t u) w_l (t,x)= \frac{1}{2} \partial_t w_l^2(t,x) + C R^{-\alpha}\tail(u_+(t);R,x_0) w_l(t,x).
\end{align}
Therefore, by following the proof of Theorem 3.6 in \cite{KaWe22b} and formalizing \eqref{eq:derivativecomputation} with the help of Steklov averages as in Section 7 in \cite{KaWe22b}, we obtain that for any $R/2 \le r \le R$, $r + \frac{\rho}{2} \le 3R/4$, $r+\rho \le R$, and $l > 0$:
\begin{align*}
\sup_{t \in I_{r+\rho}^{\ominus}(t_0)} & \int_{B_{r+\rho}(x_0)} \chi^2(t)\tau^2(x) w_l^2(t,x) \d x + \int_{I_{r+\rho}^{\ominus}(t_0)} \cE^{K(s)}_{B_{r+\rho}}(\tau w_l(s) , \tau w_l(s)) \d s\\
&\le c_1 (\rho^{-\alpha} \vee \left((r+\rho)^{\alpha} - r^{\alpha})^{-1}\right) \Vert w_l^2(s) \Vert_{L^1(I_{r+\rho}^{\ominus}(t_0) \times B_{r+\rho}(x_0))}\\
&+ c_1 \int_{I_{r+\rho}^{\ominus}(t_0)} \hspace{-0.2cm} \chi^2(s) \Vert \tau^2 w_l(s) \Vert_{L^1(B_{r+\rho}(x_0))} \\
& \qquad\qquad \times \left(\sup_{x \in B_{r + \frac{\rho}{2}}(x_0)} \int_{\R^d \setminus B_{r+\rho}(x_0)} \hspace{-0.2cm} u_+(s,y) K(s;x,y) \d y \right) \d s\\
&- C \int_{I_{r+\rho}^{\ominus}(t_0)} \chi^2(s) \Vert \tau^2 w_l(s) \Vert_{L^{1}(B_{r+\rho}(x_0))} R^{-\alpha} \tail(u_+(s);R,x_0) \d s,
\end{align*}
where $c_1 > 0$.
The last term on the right hand side stems form the second term in \eqref{eq:derivativecomputation}.
The second main novelty of our proof is the following decomposition of the tail term into an intermediate tail and a remaining tail:
{\allowdisplaybreaks
\begin{align}
&\sup_{x \in B_{r + \frac{\rho}{2}}(x_0)} \int_{\R^d \setminus B_{r+\rho}(x_0)} u_+(s,y) K(s;x,y) \d y \notag\\
&= \sup_{x \in B_{r + \frac{\rho}{2}}(x_0)} \int_{B_{R}(x_0) \setminus B_{r+\rho}(x_0)} \hspace{-0.6cm} u_+(s,y) K(s;x,y)  \d y \notag \\
&\quad + \hspace{-0.2cm} \sup_{x \in B_{r + \frac{\rho}{2}}(x_0)} \int_{\R^d \setminus B_{R}(x_0)} \hspace{-0.5cm} u_+(s,y) K(s;x,y) \d y \notag\\
&\le c_2 \rho^{-d-\alpha} \int_{B_{R}(x_0) } u_+(s,y) \d y\\
&\quad + c_2 (2-\alpha) \int_{\R^d \setminus B_{R}(x_0)} u_+(s,y) |x_0 - y|^{-d-\alpha} \d y \notag\\
&\le c_3 \rho^{-\alpha} \left( \frac{R}{\rho} \right)^d \dashint_{B_{R}(x_0) } u_+(s,y) \d y + c_3 R^{-\alpha} \tail(u_+(s);R,x_0), \label{eq:tail-split}
\end{align}
}where $c_2, c_3 > 0$, and we also used \eqref{eq:Kcomp}.
The intermediate tail takes into account values of $u$ that lie inside the solution domain, which allows us to estimate it in a second step, using the PDE. It will be treated in the De Giorgi iteration procedure just as the usual tail term was treated in the existing literature so far. The remaining tail takes into account the values of $u$ at infinity. This term will cancel with the additional term on the right hand side stemming from our choice of $w_l$ upon choosing $C := c_1c_3 > 0$. 
By combining the previous two estimates, we obtain:
\begin{align*}
&\sup_{t \in I_{r+\rho}^{\ominus}(t_0)} \int_{B_{r+\rho}(x_0)} \chi^2(t)\tau^2(x) w_l^2(t,x) \d x + \int_{I_{r+\rho}^{\ominus}(t_0)} \cE^{K(s)}_{B_{r+\rho}}(\tau w_l(s) , \tau w_l(s)) \d s\\
&~~\le c_4 (\rho^{-\alpha} \vee \left((r+\rho)^{\alpha} - r^{\alpha})^{-1}\right) \Vert w_l^2(s) \Vert_{L^1(I_{r+\rho}^{\ominus}(t_0) \times B_{r+\rho}(x_0))}\\
&~~\quad + c_4 \rho^{-\alpha} \left( \frac{R}{\rho} \right)^d \int_{I_{r+\rho}^{\ominus}(t_0)} \chi^2(s) \Vert \tau^2 w_l(s) \Vert_{L^1(B_{r+\rho}(x_0))}  \left(\dashint_{B_{R}(x_0) } \hspace{-0.2cm} u_+(s,y) \d y \right) \d s\\
&~~\le c_5 (\rho^{-\alpha} \vee \left((r+\rho)^{\alpha} - r^{\alpha})^{-1}\right) \Vert w_l^2(s) \Vert_{L^1(I_{r+\rho}^{\ominus}(t_0) \times B_{r+\rho}(x_0))}\\
&~~\quad + c_5 \rho^{-\alpha} \left( \frac{R}{\rho} \right)^d \Vert w_l \Vert_{L^1(I_{r+\rho}^{\ominus}(t_0) \times B_{r+\rho}(x_0))}  \left( \sup_{t \in I_{R}^{\ominus}(t_0)}\dashint_{B_{R}(x_0) } u_+(t,y) \d y \right),
\end{align*}
where $c_4, c_5 > 0$.
From here, we are in the position to run a well-known De Giorgi iteration scheme, choosing appropriate sequences $r_i, \rho_i, l_i$, exactly as in the proof of Theorem 3.6 in \cite{KaWe22b} (see also Lemma 4.2 in \cite{KaWe23}). The only difference to the aforementioned proof is that the $L^\infty$-tail has to be replaced by the quantity $\left( \sup_{t \in I_{R}^{\ominus}(t_0)}\dashint_{B_{R}(x_0) } u_+(t,y) \d y \right)$.\\
Consequently, after running the iteration argument, we obtain
\begin{align*}
&\sup_{I_{R/2}^{\ominus}(t_0) \times B_{R/2}(x_0)} \left(u - C R^{-\alpha}\int_{I_{R/2}^{\ominus}(t_0)} \tail(u_+(s);R,x_0) \d s \right)_+\\
&\le M := \delta \left( \sup_{t \in I_{R}^{\ominus}(t_0)}\dashint_{B_{R}(x_0) } u_+(t,x) \d x  \right)\\
&~~ + c_6 \delta^{-\frac{d+\alpha}{2\alpha}} \left( \dashint_{I_{R}^{\ominus}(t_0) \times B_R(x_0)} \left[ u(t,x) - C R^{-\alpha}\int_{t_0 - R^{\alpha}}^t \hspace{-0.2cm} \tail(u_+(s);R,x_0) \d s \right]_+^2 \d x \d t \right)^{1/2}
\end{align*}
for some constant $c_6 > 0$, which implies
\begin{align*}
\sup_{I_{R/2}^{\ominus}(t_0) \times B_{R/2}(x_0)} u_+ &\le \delta \left( \sup_{t \in I_{R}^{\ominus}(t_0)}\dashint_{B_{R}(x_0) } u_+(t,x) \d x  \right) \\
&\quad + c_7 \delta^{-\frac{d+\alpha}{2\alpha}} \left( \dashint_{I_{R}^{\ominus}(t_0) \times B_R(x_0)} u_+^2 \d x \d t \right)^{1/2}\\
&\quad + c_7 \delta^{-\frac{d+\alpha}{2\alpha}} \dashint_{I_{R}^{\ominus}(t_0)} \tail(u_+(t);R,x_0) \d t
\end{align*}
for some constant $c_7 > 0$. This proves the desired result upon estimating the first term on the right hand side using \autoref{lemma:intermediate-tail}.
\end{proof}

\section{Proof of weak Harnack inequalities with $L^1$-tail}
\label{sec:whi}

In this section, we present a proof of the improved weak parabolic Harnack inequality (see \autoref{thm:improvedwHI-parabolic}) for supersolutions. As mentioned before, in comparison to the existing results in the literature (see \autoref{lemma:wPHI}) the improved weak Harnack inequality allows us to control the nonlocal tail term of a supersolution by its infimum in a ball. For the proof we test the weak formulation with a test function of power type, as it is used for the Moser iteration for small positive exponents. Moreover, our proof relies on the well-known weak parabolic Harnack inequality:

\begin{lemma}[weak parabolic Harnack inequality]
\label{lemma:wPHI}
Let $u$ be a globally nonnegative supersolution to $\partial_t u - L_t u = 0$ in $I \times \Omega$. Let $L_t$ be as before.  Then, for every $R > 0$, $t_0 \in I$, $x_0 \in \Omega$ with $I_{4R}(t_0) \times B_{4R}(x_0) \subset I \times \Omega$, it holds
\begin{align}
\label{eq:wPHI}
\dashint_{I^{\ominus}_{R}(t_0 - R^{\alpha}) \times B_R(x_0)} u(t,x) \d x \d t \le c \inf_{I^{\oplus}_{R}(t_0) \times B_R(x_0)} u. 
\end{align}

Here, $c = c(d,\alpha_0,\lambda,\Lambda) > 0$ is a constant.
\end{lemma}

\begin{proof}
The weak parabolic Harnack inequality was proved for example in \cite{FeKa13}, \cite{CKW23}, or \cite{KaWe22a}. Note that although a slightly smaller weak solution space is considered in these references, the proof directly carries over to our more general setting. For the elliptic case, we refer the interested reader  to \cite{Kas09}, \cite{DyKa20}.
\end{proof}

Having at hand the weak parabolic Harnack inequality, we are ready to establish the improved weak parabolic Harnack inequality with $L^1$-tail, \autoref{thm:improvedwHI-parabolic}.

\begin{proof}[Proof of \autoref{thm:improvedwHI-parabolic}]
In the light of the weak parabolic Harnack inequality (see \autoref{lemma:wPHI}) it remains to prove that 
\begin{align}
\label{eq:tail-inf}
\dashint_{I_{R}^{\ominus}(t_0-R^\alpha)} \tail(u(t);R,x_0) \d t \le c \inf_{I^{\oplus}_{R}(t_0) \times B_R(x_0)} u.
\end{align}

Let $p \in (0,1)$. We denote $\U = u + \eps$ for some arbitrary $\eps > 0$. First, we claim that for every $t \in I_{R/2}^{\oplus}(t_0 -R^{\alpha})$:
\begin{align}
\label{eq:tailkey}
\begin{split}
& \cE^{K(t)}(u(t),-\tau^2 \U^{-p}(t)) \ge - c_1 R^{-\alpha} \Vert \U^{-p+1}(t)\Vert_{L^1(B_{R}(x_0))}\\
&\qquad \qquad \qquad +  c_2 \left( \int_{B_{R/2}(x_0)} \U^{-p}(t,x) \d x \right) R^{-\alpha}  \tail(u(t);R,x_0),
\end{split}
\end{align}
where $c_1,c_2 > 0$ are constants possibly depending on $p$. This estimate will be the key step towards a suitable Caccioppoli-type inequality, which will allow us to deduce \eqref{eq:tail-inf}.
To prove \eqref{eq:tailkey} take $\tau \in C_c^{\infty}(B_{3R/4}(x_0))$ such that $0 \le \tau \le 1$ and $\tau \equiv 1$ in $B_{R/2}(x_0)$. We compute for fixed $t \in I_{R/2}^{\oplus}(t_0-R^{\alpha})$ (dropping the $t$-dependence of $u$ for simplicity)
\begin{align*}
-&\cE^{K(t)}_{(B_{R}(x_0)\times B_{R}(x_0))^{c}}(u,-\tau^2 \U^{-p}) \\
&\qquad = - 2 \int_{B_{R}(x_0)} \int_{B_{R}(x_0)^{c}} (u(x)-u(y))(-\tau^2 \U^{-p}(x)) K(t;x,y) \d y \d x\\
&\qquad = 2\int_{B_{R}(x_0)}\int_{B_{R}(x_0)^{c}} \left(\frac{u(x)}{\U^{p}(x)}-\frac{u(y)}{\U^{p}(x)}\right) \tau^2(x) K(t;x,y) \d y \d x\\
&\qquad \le 2\int_{B_{R}(x_0)} \tau^2(x) \U^{-p+1}(x) \left( \int_{B_{R}(x_0)^{c}} K(t;x,y) \d y\right) \d x\\
&\qquad \quad - 2\int_{B_{R}(x_0)} \frac{\tau^2(x)}{\U^{p}(x)} \left( \int_{B_{R}(x_0)^{c}} u(y) K(t;x,y) \d y\right) \d x\\
&\qquad \le cR^{-\alpha} \Vert \U^{-p+1}\Vert_{L^1(B_{R}(x_0))} - I,
\end{align*} 
where $I$ denotes the second summand in the second to last line and we used \eqref{eq:Kcomp}. 
We estimate $I$ as follows, using the lower bound in \eqref{eq:Kcomp}:
\begin{align*}
I &\ge c \left( \int_{B_{R/2}(x_0)} \U^{-p}(x) \d x \right) \inf_{x \in B_{R/2}(x_0)} \left( \int_{B_{R}(x_0)^{c}} u(y) K(t;x,y) \d y \right)\\
&\ge c \left( \int_{B_{R/2}(x_0)} \U^{-p}(x) \d x \right) R^{-\alpha} \tail(u;R,x_0).
\end{align*}
Moreover, we recall the following estimate (see (3.8) applied with $g(s) = s^{-p}$ in \cite{KaWe22a}, or (3.4) in \cite{FeKa13}):
\begin{align*}
-\cE^{K(t)}_{B_{R}(x_0) \times B_{R}(x_0)}(u,-\tau^2\U^{-p}) &\le - c_1\cE^{K(t)}_{B_{R}(x_0) \times B_{R}(x_0)}(\tau \U^{\frac{-p+1}{2}}, \tau \U^{\frac{-p+1}{2}}) \\
&\quad +  c_2 R^{-\alpha}\Vert \U^{-p+1}\Vert_{L^1(B_{R}(x_0))}.
\end{align*}
Its proof is based on an algebraic inequality, which plays the role of a nonlocal chain rule, and on \eqref{eq:Kcomp}.
Altogether, we obtain \eqref{eq:tailkey}.

Let us take a function $\chi \in C^1(\R)$ such that $\chi \equiv 1$ in $I_{R/2}^{\oplus}(t_0 - R^{\alpha})$, $0 \le \chi \le 1$, $\Vert \chi'\Vert_{\infty} \le c R^{-\alpha}$, and $\chi(t_0 - R^{\alpha} + (R/2)^{\alpha}) = 0$. We have
\begin{align*}
-& \int_{t}^{t_0 - R^{\alpha} + (\frac{R}{2})^{\alpha}} \int_{B_R(x_0)} \chi^2(s) \partial_t u(s,x) \tau^2(x) \U^{-p}(s,x) \d x \d s\\
&= -\frac{1}{1-p} \int_{t}^{t_0 - R^{\alpha} + (\frac{R}{2})^{\alpha}} \int_{B_R(x_0)} \chi^2(s) \partial_t (\U^{-p+1})(s,x) \tau^2(x) \d x \d s\\
&\ge \frac{1}{1-p} \int_{B_R(x_0)} \chi^2(t) \U^{-p+1}(t,x) \tau^2(x) \d x \\
&\quad - \frac{c R^{-\alpha}}{1-p} \int_{t}^{t_0 - R^{\alpha} + (\frac{R}{2})^{\alpha}} \int_{B_R(x_0)} \U^{-p+1}(s,x) \tau^2(x) \d x \d s.
\end{align*}

By combining the previous estimate with \eqref{eq:tailkey}, choosing $t = t_0 - R^{\alpha}$, and using that $u$ is a weak supersolution, we obtain:
\begin{align*}
\int_{I_{R}^{\oplus}(t_0 -R^{\alpha}) } & \chi^2(s) \left( \int_{B_{R/2}(x_0)} \U^{-p}(s,x) \d x \right) R^{-\alpha}  \tail(u(s);R,x_0) \d s\\
& \le - \int_{I_{R}^{\oplus}(t_0 -R^{\alpha}) } \int_{B_R(x_0)} \chi^2(s) \partial_t u(s,x) \tau^2(x) \U^{-p}(s,x) \d x \d s\\
&\quad + \int_{ I_{R}^{\oplus}(t_0 -R^{\alpha}) }\chi^2(s) \cE^{K(s)}(u(s),-\tau^2 \U^{-p}(s)) \d s \\
&\quad + cR^{-\alpha} \Vert \U^{-p+1} \Vert_{L^1(I_{R}^{\oplus}(t_0-R^{\alpha}) \times B_{R}(x_0))}\\
& \le cR^{-\alpha} \Vert \U^{-p+1} \Vert_{L^1(I_{R}^{\oplus}(t_0 -R^{\alpha}) \times B_{R}(x_0))}.
\end{align*}
As a consequence, using Jensen's inequality ($a \mapsto a^{-1}$ is convex):
\begin{align*}
&\dashint_{I_{R/2}^{\oplus}(t_0-R^{\alpha})} \hspace{-0.4cm}  \tail(u(t);R,x_0) \d t \\
&\qquad \le c \sup_{t \in I_{R/2}^{\oplus}(t_0-R^{\alpha})}\left( \dashint_{B_{R/2}(x_0)} \U^{-p}(t,x) \d x \right)^{-1} \left( \dashint_{I_{R}^{\oplus}(t_0-R^{\alpha}) \times B_{R}(x_0)} \U^{-p+1} \right)\\
&\qquad \le c \left(\sup_{t \in I_{R/2}^{\oplus}(t_0-R^{\alpha})} \dashint_{B_{R/2}(x_0)} \U^{p}(t,x) \d x \right) \left( \dashint_{I_{R}^{\oplus}(t_0-R^{\alpha}) \times B_{R}(x_0)} \U^{-p+1} \right).
\end{align*}
Next, let us recall from the proof of Moser iteration for small positive exponents (see e.g. the estimate right before (4.9) on p.38 in \cite{KaWe22a}), that for $p \in (0,\frac{d}{d+\alpha})$ it holds
\begin{align*}
\left(\sup_{t \in I_{R/2}^{\oplus}(t_0-R^{\alpha})} \dashint_{B_{R/2}(x_0)} \U^{p}(t,x) \d x \right) \le c \left( \dashint_{I_{R}^{\oplus}(t_0-R^{\alpha}) \times B_{R}(x_0)} \U^{p} \right).
\end{align*}
As a consequence, we obtain 
\begin{align} 
\label{eq:onthewaytoharnack}
\begin{split}
&\dashint_{I_{R/2}^{\oplus}(t_0-R^{\alpha})} \tail(u(t);R,x_0) \d t \\
&\qquad \le c \left( \dashint_{I_{R}^{\oplus}(t_0-R^{\alpha}) \times B_{R}(x_0)} \U^{p} \right)\left( \dashint_{I_{R}^{\oplus}(t_0-R^{\alpha}) \times B_{R}(x_0)} \U^{-p+1} \right) \\
&\qquad \le c \inf_{I^{\oplus}_{R}(t_0+R^\alpha) \times B_R(x_0)} u\,,
\end{split}
\end{align}
where we have applied the weak parabolic Harnack inequality \autoref{lemma:wPHI} to both factors in the previous estimate with appropriately changed radii (note that $p,-p+1 \in (0,1)$). Estimate \eqref{eq:onthewaytoharnack} can be proved for adjusted time-space cylinders, thus implying \eqref{eq:tail-inf}.
\end{proof}

\begin{remark}
Note that the estimate \eqref{eq:tailkey} can be improved as follows:
\begin{align*}
&\cE^{K(t)}_{B_{R}(x_0) \times B_{R}(x_0)} (\tau \U^{\frac{-p+1}{2}}(t), \tau \U^{\frac{-p+1}{2}}(t))\\
&\qquad\qquad\qquad +  c_2 \left( \int_{B_{R/2}(x_0)} \U^{-p}(t,x) \d x \right) R^{-\alpha}  \tail(u(t);R,x_0)\\
&\qquad\le c_1 |p-1|\cE^{K(t)}(u(t),-\tau^2 \U^{-p}(t))\\
&\qquad\qquad\qquad\quad + c_2(1 \vee |p-1|) R^{-\alpha} \Vert \U^{-p+1}(t)\Vert_{L^1(B_{R}(x_0))}.
\end{align*}
This is an improved version of the Caccioppoli-type inequality for small positive exponents (see Lemma 3.5 in \cite{KaWe22a}). When applied to weak supersolutions, the first term on the left hand side can be used to run a Moser iteration scheme, which is a main ingredient in the proof of the weak parabolic Harnack inequality. The second term is used to prove the improved tail estimate, as we have done in the aforementioned proof. In this sense, the above estimate captures local and nonlocal information at the same time, similar to the improved De Giorgi classes in \cite{Coz17}.
\end{remark}

\section{Proof of the main results}
\label{sec:main-proofs}
Having at hand both, the improved local boundedness estimate \autoref{thm:locbdL1tail} and the improved weak Harnack inequality \autoref{thm:improvedwHI-parabolic}, we are now in the position to prove our main results, i.e. \autoref{thm:fHI}, \autoref{cor:fPHI}, and \autoref{thm:HRE}. We provide a counterexample in \autoref{ex:counterexample} to Hölder estimates for weak solutions with merely finite $L^1$-tails.

\subsection{Full Harnack inequalities}

We start by explaining how to establish the full Harnack inequality \autoref{thm:fHI}:

\begin{proof}[Proof of \autoref{thm:fHI}]
Note that as a consequence of \eqref{eq:locbdL1tail} and Young's inequality, for every $\delta > 0$ there exists a constant $c(\delta) > 0$ such that
\begin{align*}
\sup_{I^{\ominus}_{R/2}(t_0) \times B_{R/2}(x_0)} \hspace{-0.2cm} u & \le \delta \hspace{-0.2cm} \sup_{I^{\ominus}_R(t_0) \times B_R(x_0)} \hspace{-0.2cm} u \\
&\quad + c(\delta) \dashint_{I^{\ominus}_R(t_0) \times B_R(x_0)} \hspace{-0.4cm} u(t,x) \d x \d t  + c\dashint_{I^{\ominus}_R(t_0)} \hspace{-0.2cm} \tail(u(t);R,x_0) \d t.
\end{align*}
Since $u$ locally bounded in $I \times \Omega$ due to \autoref{thm:locbdL1tail}, we can absorb the first term on the right hand side by using a classical covering argument as in the proof of Theorem 6.9 in \cite{Coz17} and \autoref{lemma:it} (see also the proof of Theorem 6.2 in \cite{KaWe22b}). This yields:
\begin{align}
\label{eq:proof-fHI-locbd}
\sup_{I^{\ominus}_{R/2}(t_0) \times B_{R/2}(x_0)} \hspace{-0.2cm} u \le c \dashint_{I^{\ominus}_R(t_0) \times B_R(x_0)} \hspace{-0.4cm} u(t,x) \d x \d t  + c\dashint_{I^{\ominus}_R(t_0)} \hspace{-0.2cm} \tail(u(t);R,x_0) \d t.
\end{align}
From here, the full Harnack inequality \eqref{eq:fPHI} follows after application of \autoref{thm:improvedwHI-parabolic} to the right hand side of \eqref{eq:proof-fHI-locbd}. Note that this step requires a modification of \eqref{eq:improvedwHI} where time-space cylinders are adjusted appropriately.

\end{proof}

Let us explain how to deduce \autoref{cor:fPHI} from \autoref{thm:fHI}:

\begin{proof}[Proof of \autoref{cor:fPHI}]
The proof follows by decomposing $u = u_+ - u_-$ and observing that $\partial_t u_+ - L_t u_+ = f$ in $I \times \Omega$, where $f = L_t u_- \in L^{1,\infty}_{t,x}(I_{4R}(t_0) \times B_{4R}(x_0))$ and for any $t \in I_{2R}(t_0)$:
\begin{align*}
\Vert f(t) \Vert_{L^{\infty}(B_{4R}(x_0))} \le c \tail(u_-(t); 4R, x_0).
\end{align*}
All proofs of this article can be generalized to equations with such source terms. The argument how to extend \autoref{thm:locbdL1tail} accordingly is described in the proof of \autoref{lemma:locbdsuptail}. In order to extend \autoref{thm:improvedwHI-parabolic}, we take $\U = u + \eps + \int_{t_0 - R^{\alpha}}^t \Vert f(s) \Vert_{L^{\infty}(B_{4R}(x_0))} \d s$ in the proof. This leads to a cancellation, similar to the proof of \autoref{thm:locbdL1tail} in \autoref{sec:lb_second-proof}. Note that \autoref{lemma:wPHI} can be generalized in the same way.
\end{proof}

\subsection{H\"older regularity estimates}

Finally, we give a proof of the parabolic H\"older regularity estimate \autoref{thm:HRE}. Moreover, we give a counterexample in case the tail is only assumed to be $L^1$ in time (see \autoref{ex:counterexample}).\\
The first step in the proof of \autoref{thm:HRE} is to establish a suitable growth lemma, reminiscent of Lemma 6.3 in \cite{Coz17} and Lemma 5.1 in \cite{KaWe22a}.

We introduce the following notation:
\begin{align*}
D(t_0,x_0,R) &= (t_0 - 2R^{\alpha} , t_0) \times B_{2R}(x_0),\\
\widehat{D}(t_0,x_0,R) &= (t_0 - 2 R^{\alpha} , t_0 ) \times B_{3R}(x_0),\\
D_{\ominus}(t_0,x_0,R) &= (t_0 - 2 R^{\alpha} , t_0 - 2R^{\alpha} + (R/2)^{\alpha}) \times B_{R/2}(x_0),\\
D_{\oplus}(t_0,x_0,R) &= (t_0 - (R/2)^{\alpha} , t_0) \times B_{R/2}(x_0).
\end{align*}

\begin{lemma}[growth lemma]
\label{lemma:growth}
Assume that $K$ satisfies \eqref{eq:Kcomp}. Let $H > 0$. Then, for every $R > 0$, $t_0 \in I$, $x_0 \in \Omega$ with $\widehat{D}(t_0,x_0,R) \Subset I \times \Omega$, and any $\sigma \in (0,1)$ there exists $\delta = \delta(d,\alpha_0,\lambda,\Lambda,\sigma) \in (0,1/8]$ such that if the following hold true
\begin{itemize}
\item[(i)] $u \ge 0$ in $\widehat{D}(t_0,x_0,R)$,
\item[(ii)] $u$ is a weak supersolution to $\partial_t u - L_t u = 0$ in $D(t_0,x_0,R)$,
\item[(iii)] $| D_{\ominus}(t_0,x_0,R) \cap \{ u \ge H \}| \ge \sigma |D_{\ominus}(t_0,x_0,R)|$,
\item[(iv)] $\dashint_{I^{\ominus}_{2^{1/\alpha}R}(t_0)} \tail(u_-(t);2R,x_0) \d t \le H\delta$,
\end{itemize}
then it holds $u \ge H\delta$ in $D_{\oplus}(t_0,x_0,R)$.
\end{lemma}

\begin{proof}[Proof of \autoref{lemma:growth}]
By the same argument as in the proof of \autoref{cor:fPHI}, we can prove the following weak Harnack inequality for $u$, using $(i)$ and $(ii)$:
\begin{align*}
\inf_{D_{\oplus}(t_0,x_0,R)} u \ge c_1 \left( \dashint_{D_{\ominus}(t_0,x_0,R)} \hspace{-0.2cm} u(t,x) \d x \d t \right) - c_2\dashint_{I^{\ominus}_{2^{1/\alpha}R}(t_0)} \hspace{-0.2cm} \tail(u_-(t);2R,x_0) \d t
\end{align*}
for some constants $c_1,c_2 > 0$.
By application of $(iii)$ and $(iv)$, this implies
\begin{align*}
\inf_{D_{\oplus}(t_0,x_0,R)} u \ge c_1 H \sigma - c_2 H\delta.
\end{align*}
This yields the desired result upon choosing $\delta = \frac{c_1 \sigma}{1 + c_2} \wedge \frac{1}{8}$.
\end{proof}

Now, we are in the position to establish the parabolic H\"older regularity estimate.  Since the proof is mostly a parabolic rewriting of the proof of Theorem 6.4 in \cite{Coz17}, we will only sketch the main arguments. The only difference to the proof in \cite{Coz17} is the treatment of the parabolic tail terms (see \eqref{eq:Holder-parabolic-tail}).

\begin{proof}[Proof of \autoref{thm:HRE}]
Let $\delta \in (0,1/8]$ be the constant from \autoref{lemma:growth}, given $\sigma = 1/2$. The idea of the proof is to construct two sequences $(M_j)$, and $(m_j)$ that are non-decreasing, and non-increasing, respectively, and to find a small parameter $\gamma \in (0,1)$ such that for every $j \in \N$
\begin{align}
\label{eq:HRE-goal}
m_j \le u \le M_j ~~ \text{ in } \widehat{D}(t_0,x_0,\nu^{-j} R), \quad \text{ and } M_j - m_j = L \nu^{-\gamma j},
\end{align}
where $\nu = 6 \vee 2^{1 + \frac{1}{\alpha_0}}$, and for a constant $C_0 > 0$ to be determined later, we set
\begin{align*}
L := C_0 \Vert u \Vert_{L^{\infty}(\widehat{D}(t_0,x_0,R))} + \left(\dashint_{I^{\ominus}_{2R}(t_0)} \tail(u(t);R/2,x_0)^{1+\eps} \d t \right)^{\frac{1}{1+\eps}}.
\end{align*}
Note that \eqref{eq:HRE-goal} yields the desired result by definition of $L$ and upon application of \autoref{thm:locbdL1tail} in order to estimate $\Vert u \Vert_{L^{\infty}(\widehat{D}(t_0,x_0,R))}$ from above.

We will choose $\gamma \in (0,1)$ such that $\gamma < \log_{\nu}\left(\frac{2}{2-\delta} \right)$, $\widehat{D}(t_0,x_0,\nu^{-1} R) \subset D_{\oplus}(t_0,x_0,R)$ and
\begin{align}
\label{eq:HRE-tailsmall}
\int_{3}^{\infty} \frac{s^{\gamma} - 1}{s^{1+\alpha_0}} \d s \le \frac{\delta}{8 (1 \vee |B_1|)}, \qquad \gamma \le \frac{\alpha \eps}{2(1+\eps)}.
\end{align}
Clearly, by setting $M_j = \nu^{-\gamma j}L/2$ and $m_j = -\nu^{-\gamma j}L/2$, \eqref{eq:HRE-goal} holds true for every $j \le j_0$ if we choose $C_0 \ge 2 \nu^{\gamma j_0}$. Later, we will choose $j_0$ appropriately. To show that \eqref{eq:HRE-goal} holds true for $j > j_0$, let us do a proof by induction over $j$ and assume that \eqref{eq:HRE-goal} holds true up to some fixed $j \ge j_0$. We will now explain how to construct $M_{j+1}$ and $m_{j+1}$. To do so, let us assume without loss of generality that
\begin{align}
\label{eq:HRE-cases}
|D_{\ominus}(t_0,x_0,\nu^{-j} R) \cap \{ u \ge m_j + (M_j - m_j)/2 \}| \ge \frac{1}{2} |D_{\ominus}(t_0,x_0,\nu^{-j} R)|
\end{align}
and define $v = u - m_j$. Note that if \eqref{eq:HRE-cases} does not hold true, we would proceed analogously, but define $v = M_j - u$. We set $H := (M_j - m_j)/2 = \nu^{-\gamma j} L /2$. Our goal is to apply the growth lemma (see \autoref{lemma:growth}) with $u := v$, $\sigma = 1/2$, and $R := \nu^{-j}R$ in order to deduce that $v \ge \delta H$ in $D_{\oplus}(t_0,x_0,\nu^{-j}R)$, which implies that $u \ge m_j + \nu^{-\gamma j}\delta L/2$ in $D_{\oplus}(t_0,x_0,\nu^{-j}R) \supset \widehat{D}(t_0,x_0,\nu^{-(j+1)}R)$. By defining $M_{j+1} = M_j$ and $m_{j+1} = m_j + \nu^{-\gamma j}\delta L/2$, we verify \eqref{eq:HRE-goal} for $m_{j+1}$ and $M_{j+1}$, as desired.\\
It remains to verify the assumptions of \autoref{lemma:growth}. Note that $(i)$, $(ii)$, and $(iii)$ hold true by construction. In order to verify $(iv)$, one proceeds exactly as in the proof of Theorem 4.2 in \cite{Coz17}, using that $v(t,y) \ge -2H \left( \left( \frac{|y|}{\nu^{-j}R} \right)^{\gamma} - 1\right)$ for $y \in B_{3R} \setminus B_{3\nu^{-j}R}$ and $v(t,y) \ge - |u(t,y)| - L/2$ in $\R^n \setminus B_{3R}$ for every $t \in I_{2^{1/\alpha}\nu^{-j} R}^{\ominus}(t_0)$. Therefore,
\begin{align}
\label{eq:Holder-parabolic-tail}
\begin{split}
& \Bigg(\dashint_{I^{\ominus}_{2^{1/\alpha} \nu^{-j}R}} \tail(v_-(t);2 \nu^{-j}R,x_0)^{1+\eps} \d t \Bigg)^{\frac{1}{1+\eps}} \\
&\qquad\qquad\le H (\nu^{-j}R)^{-\alpha} \int_{B_{3R} \setminus B_{3 \nu^{-j}R}} \left( \left( \frac{|y|}{\nu^{-j}R} \right)^{\gamma} - 1\right) |y|^{-d-\alpha} \d y\\
&\qquad\qquad\quad + c_1 \nu^{-j \alpha} \left(\dashint_{I^{\ominus}_{2^{1/\alpha} \nu^{-j}R}} \left( \tail(u(t);3R,x_0) + L \right)^{1+\eps}\d t \right)^{\frac{1}{1+\eps}}\\
&\qquad\qquad=: J_1 + J_2
\end{split}
\end{align}
for some $c_1 > 0$.
By \eqref{eq:HRE-tailsmall} and a change of variables, we see that $J_1 \le \delta H/2$. Moreover, by definition $J_2 \le c_1 \nu^{-\frac{j\alpha \eps}{1+\eps}} L + c_1 \nu^{-j\alpha} L =  2c_1 \left(\nu^{-\frac{j\alpha \eps}{1+\eps}} + \nu^{-j\alpha}\right)\nu^{\gamma j} H$. \\
By the choice of $\gamma$, and in particular by \eqref{eq:HRE-tailsmall} and upon choosing $j_0 \in \N$ so large that $J_2 \le \delta H/2$ (for more details on the underlying computation, we refer once more to \cite{Coz17}), this implies that $\tail(v(t);2 \nu^{-j}R,x_0) \le \delta$ for every $t \in I_{2^{1/\alpha} \nu^{-j} R}^{\ominus}(t_0)$. Since this yields $(iii)$, we are now able to apply \autoref{lemma:growth}, as desired, and to conclude the proof.
\end{proof}

We complete this section by providing a counterexample to the parabolic H\"older regularity estimate if the assumption $\tail(u) \in L^{1+\eps}_{loc}(I)$ is dropped. This proves the sharpness of the condition in \autoref{thm:HRE} in the scale of $L^p$-spaces.

\begin{example}
\label{ex:counterexample}
Let us consider a solution $u$ to
\begin{align*}
\begin{cases}
\partial_t u + (-\Delta)^{\alpha/2} u &= 0 ~~ \text{ in } (-1,1) \times B_1,\\
u(-1) &\equiv 0 ~~ \text{ in } \R^d,\\
u &\equiv g ~~ \text{ in } (-1,1) \times \R^d \setminus B_1,
\end{cases}
\end{align*}
where $g(t) \equiv 0$ for $t \in (-1,0)$ and $g(t,x) = \delta f(t) + f'(t) \mathbbm{1}_{B_3 \setminus B_2}(x)$ for $t \in [0,1)$.\\
We will choose $f : (-1,1) \to \R$ to be a function satisfying $f \not\in C^{\gamma}_{loc}(-1,1)$, $f+f' \in L^1_{loc}(-1,1) \setminus L^{1+\gamma}_{loc}(-1,1)$ for every $\gamma > 0$ and $f'(t) \ge 0$ for any $t \in [0,1)$. Moreover, we will choose $\delta > 0$ so small that $\delta + (-\Delta)^{\alpha/2} \mathbbm{1}_{B_3 \setminus B_2}(x) \le 0$ for every $x \in B_1$, which implies that $g$ is a subsolution to $\partial_t g + (-\Delta)^{\alpha/2} g = 0$ in $(-1,1) \times B_1$ since
\begin{align*}
\partial_t g + (-\Delta)^{\alpha/2} g = \delta f' + f'' \mathbbm{1}_{B_3 \setminus B_2} + f' (-\Delta)^{\alpha/2}  \mathbbm{1}_{B_3 \setminus B_2} \le 0~~ \text{ in } [0,1) \times B_1.
\end{align*}
As a consequence, by the comparison principle, $u(t) \equiv 0$ for $t \in (-1,0]$ and $u(t) \ge g(t)$ in $B_1$ for any $t \ge 0$. By the choice of $f$, this implies that $u$ is not H\"older continuous. Moreover, note that $\tail(u(t);1,0) \asymp f(t) + f'(t)$ and is therefore in $L^1_{loc}(-1,1) \setminus L^{1+\gamma}_{loc}(-1,1)$ for every $\gamma > 0$. A possible choice of a function $f$ satisfying the aforementioned criteria is $f(t) = (\log t)^{-2}$.
\end{example}

\section{Extension of our techniques} 
\label{sec:ext}

In this section we explain how the techniques developed in this article can be modified to treat jumping kernels that do not satisfy the coercivity assumption \eqref{eq:Kcomp}. Let us split \eqref{eq:Kcomp} into two parts:
\begin{align}
\label{eq:Klower}\tag{$K_{\ge}$}
\lambda (2-\alpha) |x-y|^{-d-\alpha} \le & K(t;x,y),\\
\label{eq:Kupper}\tag{$K_{\le}$}
&K(t;x,y) \le \Lambda (2-\alpha) |x-y|^{-d-\alpha}
\end{align}
for every $t \in I$ and $x,y \in \R^d$.
In fact, regularity results for variational solutions to nonlocal equations have been obtained in several works, where the jumping kernel does not satisfy the pointwise bounds of \eqref{eq:Kcomp}.

The results in this section are split into two parts. First, we prove the full parabolic Harnack inequality (see \autoref{thm:fHI}) for a large class of operators with jumping kernels that do not satisfy \eqref{eq:Kcomp}. Second, we investigate singular nonlocal operators, i.e. operators possessing a jumping measure that is not absolutely continuous with respect to the Lebesgue measure. The parabolic Harnack inequality is known to fail for singular nonlocal operators (see \cite{BaCh10}). However, using the techniques of this article,  we are able to establish an estimate reminiscent of the parabolic Harnack inequality  involving suitable tail terms.

\subsection{Full parabolic Harnack inequalities under weaker assumptions}

In this section we establish the full parabolic Harnack inequality under substantially weaker assumptions than \eqref{eq:Kcomp}. The main result of this section is \autoref{thm:fHI_sharp-ass}.

A typical relaxation of the lower bound \eqref{eq:Klower}, is to assume that a Poincar\'e inequality and a Sobolev inequality hold true, i.e. that there is $\lambda > 0$ such that for every ball $B_{r+\rho} \subset \Omega$ with $0 < \rho < r$, $t \in I$, and $v \in L^2(B_r)$:
\begin{align}
\label{eq:Poinc}\tag{$\text{Poinc}$}
\lambda \int_{B_r} \left(v(x) - [v]_{B_r}\right)^2 \d x &\le r^{\alpha} \cE^{K(t)}_{B_r}(v,v),\\
\label{eq:Sob}\tag{$\text{Sob}$}
\lambda \Vert v^2 \Vert_{L^{\frac{d}{d-\alpha}}(B_r)} &\le \cE^{K(t)}_{B_{r+\rho}}(v,v) + \rho^{-\alpha}\Vert v^2\Vert_{L^{1}(B_{r+\rho})},
\end{align}
where $[v]_{B_r} = \dashint_{B_r} v(x) \d x$.

\begin{remark}
Note that \eqref{eq:Poinc} and \eqref{eq:Sob} are satisfied for the fractional Sobolev seminorm. In particular, a sufficient condition for \eqref{eq:Poinc} and \eqref{eq:Sob}  to hold true is the following:
\begin{align}
\label{eq:coercive}
\cE^{K(t)}_{B_r}(v,v) \ge c [v]_{H^{\alpha/2}(B_r)}^2 ~~ \forall t \in I, ~~ \forall r > 0.
\end{align}
For the investigation of the coercivity assumption \eqref{eq:coercive} and the discussion of suitable examples, we refer to \cite{BKS19}, \cite{ChSi20}, \cite{DyKa20}, \cite{KaWe22a}, and \cite{CKW23}.
\end{remark}

The improved local boundedness estimate holds true under the following assumptions:

\begin{lemma}
\label{lemma:locbd-gen}
Assume that $K$ satisfies \eqref{eq:Kupper}, \eqref{eq:Poinc}, and \eqref{eq:Sob}. Then, the improved local boundedness estimate \eqref{eq:locbdL1tail} holds true for any weak subsolution $u$ to $\partial_t u - L_t u = 0$ in $I \times \Omega$ for every $R > 0$, $t_0 \in I$, $x_0 \in \Omega$ with $I_{4R}(t_0) \times B_{4R}(x_0) \subset I \times \Omega$. In particular, $u_+ \in L^{\infty}(I^{\ominus}_{R/2}(t_0) \times B_{R/2}(x_0))$. 
\end{lemma}

\begin{proof}
In Theorem 3.6 in \cite{KaWe22b} (see also Lemma 4.2 in \cite{KaWe23}) we have proved that the local boundedness estimate with $L^\infty$-tail, i.e. \autoref{lemma:locbdsuptail} holds true under \eqref{eq:Kupper}, \eqref{eq:Poinc}, and \eqref{eq:Sob}. Moreover, note that the proof of \autoref{lemma:intermediate-tail} only relies on \eqref{eq:Kupper} and therefore still remains true. Since also the rest of the proof of \autoref{thm:locbdL1tail} in \autoref{sec:lb_first-proof} only relies on \eqref{eq:Kupper}, we conclude that \eqref{eq:locbdL1tail} remains true. 
\end{proof}

Moreover, instead of the upper bound \eqref{eq:Kupper} it is possible to assume that there is $\Lambda > 0$ such that for every $\rho > 0$
\begin{align}
\label{eq:cutoff}\tag{$\text{Cutoff}$}
\sup_{(t,x) \in I \times \Omega} \int_{\R^d \setminus B_{\rho}(x)} K(t;x,y) \d y \le \Lambda \rho^{-\alpha},
\end{align}

and still obtain the weak Harnack inequality (see \autoref{lemma:wPHI}).

\begin{lemma}
\label{lemma:wPHI-gen}
Assume that $K$ satisfies \eqref{eq:cutoff}, \eqref{eq:Poinc}, and \eqref{eq:Sob}. Then, the weak parabolic Harnack inequality \eqref{eq:wPHI} holds true for any globally nonnegative supersolution $u$ to $\partial_t u - L_t u = 0$ in $I \times \Omega$.
\end{lemma}

\begin{proof}
This result is proved in Theorem 1.1 in \cite{KaWe22a}. See also \cite{FeKa13}.
\end{proof}

An important difficulty when dropping the pointwise lower bound \eqref{eq:Klower} is that it does not suffice to consider $\tail(u;R,x_0)$, as before. In fact, we need to define another tail term which is given with respect to the jumping kernel $K$ itself:
\begin{align*}
\tail_{K(t)}(v;r,R,x_0) = \sup_{x \in B_{r}(x_0)} \int_{\R^d \setminus B_{R}(x_0)} |v(y)| K(t;x,y) \d y.
\end{align*}
Note that $\tail_{K(t)}$ was already introduced in \cite{KaWe22b}. Clearly, because $K(x,y)$ and $|x-y|^{-d-\alpha}$ are no longer comparable, also $\tail$ and $\tail_{K(t)}$ are not comparable.

In order to prove the full parabolic Harnack inequality, involving an estimate for $\tail_{K(t)}$, and therefore also the full parabolic Harnack inequality, we additionally need to impose the following assumption, which states that $K$ does not allow for oscillating long jumps: Assume that there is $c > 0$ such that for every $x,y \in \R^d$, $t \in I$, and every $r \le \left(\frac{1}{4} \wedge \frac{\vert x-y \vert}{4}\right)$ with $B_r(x) \subset \Omega$ it holds
\begin{align}
\label{UJS}\tag{$\text{UJS}$}
K(t;x,y) &\le c \dashint_{B_r(x)} K(t;z,y) \d z.
\end{align}
The assumption \eqref{UJS} is common in the literature about Harnack inequalities for nonlocal operators with general kernels (see e.g. \cite{Sch20}, \cite{CKW20}, \cite{KaWe22b}). It was introduced for the first time in \cite{BBK09}.

Now, we are in the position to state and prove the main result of this section, namely the following improvement of the full parabolic Harnack inequality (see \autoref{thm:fHI}).

\begin{theorem}
\label{thm:fHI_sharp-ass}
Assume that $K$ satisfies \eqref{eq:Kupper}, \eqref{eq:Poinc}, \eqref{eq:Sob}, and \eqref{UJS}. Then, the full parabolic Harnack inequality \eqref{eq:fPHI} holds true for any globally nonnegative solution $u$ to $\partial_t u - L_t u = 0$ in $I \times \Omega$.
\end{theorem}

\begin{remark}
Note that the above theorem is almost sharp in the sense of the equivalent characterization of the full parabolic Harnack inequality in \cite{CKW20} for time-homogeneous kernels on doubling metric spaces. Indeed, using a probabilistic approach, in \cite{CKW20} it is proved that a parabolic Harnack inequality is equivalent to \eqref{eq:Kupper}, \eqref{eq:Poinc}, and \eqref{UJS}. We conjecture that it is also possible to get rid of the assumption \eqref{eq:Sob} in the proof of \autoref{thm:fHI_sharp-ass} by application of energy methods. In fact, the only missing ingredient is the proof of the weak parabolic Harnack inequality \eqref{eq:wPHI} without using \eqref{eq:Sob}. We believe this to be possible using an approach via the De Giorgi iteration technique and the Faber-Krahn inequality (which follows from \eqref{eq:Poinc}).
\end{remark}

\begin{proof}[Proof of \autoref{thm:fHI_sharp-ass}]
As before, the proof relies on an improved local boundedness estimate and an improved weak Harnack inequality. However, recall that we have to work with $\tail_{K(t)}$ due to the fact that \eqref{eq:Kcomp} are not satisfied. Since $\tail_{K(t)}$ does not behave well with respect to scaling, we only expect weaker versions of these estimates to hold true (see \eqref{eq:weak-tail-inf} and \eqref{eq:key_improved-lb}):

Step 1: Let us first explain how to estimate the $L^1$-tail by the infimum. We will establish an estimate similar to \eqref{eq:tail-inf} (see \eqref{eq:weak-tail-inf}). We start by proving the following replacement of \eqref{eq:tailkey}: 
\begin{align}
\label{eq:tailkey2}
\begin{split}
&\cE^{K(t)}(u(t),-\tau^2 \U^{-p}(t)) \ge - c_1 R^{-\alpha} \Vert \U^{-p+1}(t)\Vert_{L^1(B_{R}(x_0))}\\
&\qquad\qquad\qquad + c_2  R^{d} \inf_{x \in B_{R}(x_0)} \U^{-p}(t,x) \tail_{K(t)}(u(t);3R/4,R,x_0).
\end{split}
\end{align}
In order to do so, we carry out the same computation as in the proof of \autoref{thm:improvedwHI-parabolic} but estimate $I$ as follows for any $z \in B_{\frac{3R}{4}}$, using \eqref{UJS} and that $\tau$ satisfies $\tau \ge c$ in $B_{\frac{7R}{8}}$:
\begin{align*}
I &= 2\int_{B_{R}(x_0)} \frac{\tau^2(x)}{\U^{p}(x)} \left( \int_{B_{R}(x_0)^{c}} u(y) K(t;x,y) \d y\right) \d x\\
&\ge c R^d \inf_{B_{\frac{R}{8}(z)}} \U^{-p} \left( \dashint_{B_{\frac{R}{8}}(z)} \int_{B_{R}(x_0)^{c}} u(y) K(t;x,y) \d y \d x\right)\\
&\ge c R^d \inf_{B_R}  \U^{-p}  \left( \int_{B_{R}(x_0)^c} u(y) K(t;z,y) \d y \right).
\end{align*}

As a consequence, using Young's inequality, and also the weak Harnack inequality (see \autoref{lemma:wPHI-gen})
\begin{align}
\label{eq:weak-tail-inf}
\begin{split}
& \dashint_{I_{R/2}^{\oplus}(t_0-R^{\alpha})} \tail_{K(t)}(u(t);3R/4,R,x_0) \d t \\
& \qquad\qquad \le c \left( \sup_{I_{R/2}^{\oplus}(t_0-R^{\alpha}) \times B_{R}(x_0)} \U^p \right) \left( \dashint_{I_{R}^{\oplus}(t_0-R^{\alpha}) \times B_{R}(x_0)} \U^{-p+1} \right)\\
&\qquad \qquad \le \delta \sup_{I_{R/2}^{\oplus}(t_0-R^{\alpha}) \times B_{R}(x_0)} \U + c(\delta) \inf_{I^{\oplus}_{R}(t_0+R^\alpha) \times B_R(x_0)} \U.
\end{split}
\end{align}

Step 2: Next, we aim to prove an improved local boundedness estimate (see \eqref{eq:key_improved-lb}). We observe that the proof in \autoref{sec:lb_second-proof} can easily be modified so that we obtain for any $\delta > 0$:
\begin{align}
\label{eq:key_improved-lb}
\begin{split}
&\sup_{I_{R/2}^{\ominus}(t_0) \times B_{R/2}(x_0)} u_+\\
&\qquad\le \delta \left( \sup_{t \in I_{R}^{\ominus}(t_0)}\dashint_{B_{R}(x_0) } u_+(t,x) \d x  \right) + c \delta^{-\frac{d+\alpha}{2\alpha}} \left( \dashint_{I_{R}^{\ominus}(t_0) \times B_R(x_0)} \hspace{-0.4cm} u_+^2  \right)^{1/2}\\
&\qquad\quad + c \delta^{-\frac{d+\alpha}{2\alpha}} \dashint_{I_{R}^{\ominus}(t_0)} \tail_{K(t)}(u_+(t);3R/4,R,x_0) \d t.
\end{split}
\end{align}
Indeed, by redefining
\begin{align*}
w_l(t,x) = \left(u(t,x) - C \int_{t_0-(r+\rho)^{\alpha}}^t \tail_{K(s)}(u_+(s);3R/4,R,x_0) \d s - l \right)_+,
\end{align*}
and not using \eqref{eq:Kupper} on the second summand in \eqref{eq:tail-split} in order to keep $\tail_{K(s)}(u_+(s);3R/4,R,x_0)$, we directly obtain \eqref{eq:key_improved-lb} upon following the same arguments as before. Note that \eqref{eq:key_improved-lb} can also be proved for adjusted time-space cylinders.

Step 3: Since $u \ge 0$ in $I_{2R}(t_0) \times B_{2R}(x_0)$, we can further estimate the following version of \eqref{eq:key_improved-lb} with accordingly adjusted time-space cylinders, using Young's inequality, to obtain:
\begin{align}
\label{eq:key_improved-lb2}
\begin{split}
&\sup_{I_{R/4}^{\ominus}(t_0 - R^{\alpha} + (\frac{R}{2})^{\alpha}) \times B_{R/4}(x_0)} u \\
&\qquad\le \delta \sup_{I_{R/2}^{\ominus}(t_0 - R^{\alpha} + (\frac{R}{2})^{\alpha} ) \times B_{R/2}(x_0)} u + c(\delta) \left( \dashint_{I_{R/2}^{\oplus}(t_0-R^{\alpha}) \times B_{R/2}(x_0)} u  \right)\\
&\qquad\quad +  c(\delta) \dashint_{I_{R/2}^{\oplus}(t_0-R^{\alpha})} \tail_{K(t)}(u(t);3R/4,R,x_0) \d t\\
&\qquad\le (1+c)\delta \sup_{I_{R}^{\ominus}(t_0 - R^{\alpha} + (\frac{R}{2})^{\alpha}) \times B_{R}(x_0)} u + c(\delta) \inf_{I^{\oplus}_{R}(t_0) \times B_R(x_0)} u.
\end{split}
\end{align}
Note that in the last step, we used \eqref{eq:weak-tail-inf} (after taking the limit $\eps \searrow 0$) and \autoref{lemma:wPHI-gen}. Moreover, we made use of the fact that $I_{R/2}^{\ominus}(t_0 - R^{\alpha} + (\frac{R}{2})^{\alpha}) = I_{R/2}^{\oplus}(t_0 - R^{\alpha})$.
From here, using that $u$ is locally bounded in $I \times \Omega$ by \autoref{lemma:locbd-gen}, we deduce the full parabolic Harnack inequality \eqref{eq:fPHI} by a classical covering and iteration argument based on \autoref{lemma:it} (see \cite{KaWe22b}, \cite{Coz17}, or \cite{DKP14}). 
\end{proof}

\subsection{Harnack-type inequalities for singular nonlocal operators}\ 

Let us end this section by commenting on the extension of our method to nonlocal operators with singular jumping measures $\mu(t;x,\d y)$ instead of $K(t;x,y)\d y$. Here, ``singular'' means that $\mu$ does not possess a density with respect to Lebesgue measure. The most prominent example of such a jumping measure is given by
\begin{align*}
\mu_{\text{axes}}(t;x,\d y)=\mu_{\text{axes}}(x,\d y) = (2-\alpha) \sum_{i = 1}^d |x_i - y_i|^{-1-\alpha} \d y_i \prod_{j \neq i}^d \delta_{x_j}(\d y_j).
\end{align*}
Although $\mu_{\text{axes}}$ is singular, the corresponding operator shares several properties with operators that have nice jumping kernels. Indeed, the corresponding operator $L_{\mu_{\text{axes}}}$ (or the energy form $\cE_{\mu_{\text{axes}}}$) also satisfies \eqref{eq:Poinc}, \eqref{eq:Sob}, and \eqref{eq:cutoff}. Consequently, the nonlocal De Giorgi-Nash-Moser theory carries over to this class of operators and weak supersolutions to the corresponding parabolic equation satisfy the weak parabolic Harnack inequality \eqref{eq:wPHI} (see for instance \cite{KaSc14}, \cite{DyKa20}, \cite{ChKa20}, and \cite{CKW23}). \\
On the other hand, most importantly, due to the singularity of the jumping kernel weak solutions do not satisfy the full Harnack inequality \eqref{eq:fPHI}, as it was observed in \cite{BaCh10} in the elliptic case. However, as an application of the technique developed in this article, we can show that a Harnack inequality involving certain tail terms does hold true (see \autoref{thm:fHI_singular}).

We define the following tail term, which can be seen as a hybrid between $\tail$ and $\tail_{K(t)}$:
\begin{align*}
\tail_{\text{axes}}(v;R,x_0) = \sup_{x \in B_R(x_0)} R^{\alpha} \sum_{i = 1}^d \hspace{-3cm}\int\limits_{\hspace{3cm}\{ h \in \R : (x + h e_i) \in \R^d \setminus B_R(x_0) \}} \hspace{-3cm} |v(x + h e_i)| |(x_0)_i - (x_i + h)|^{-1-\alpha} \d h.
\end{align*}

We have the following result, which is new even in the elliptic case:

\begin{theorem}
\label{thm:fHI_singular}
For any bounded, globally nonnegative weak solution $u$ to $\partial_t u - L_{\mu_{\text{axes}}} u = 0$ in $I \times \Omega$, for every $R > 0$, $t_0 \in I$, $x_0 \in \Omega$ with $I_{4R}(t_0) \times B_{4R}(x_0) \subset I \times \Omega$, we have 
\begin{align*}
\sup_{I^{\ominus}_R(t_0 - R^{\alpha}) \times B_R(x_0)} u \le c \inf_{I^{\oplus}_R(t_0) \times B_R(x_0)} u + c\dashint_{I^{\ominus}_{R}(t_0)} \hspace{-0.2cm} \tail_{\text{axes}}(u(t);R,x_0) \d t.
\end{align*}
Here, $c = c(d,\alpha_0,\lambda,\Lambda) > 0$ is a constant.
\end{theorem}

\begin{proof}
Note that it suffices to show the following variant of \eqref{eq:locbdL1tail}
\begin{align}
\label{eq:locbdL1tail_sing}
\begin{split}
\sup_{I^{\ominus}_{R/4}(t_0) \times B_{R/4}(x_0)} \hspace{-0.2cm} u &\le c \left(\dashint_{I^{\ominus}_{R/2}(t_0) \times B_{R/2}(x_0)} \hspace{-0.2cm} u^2(t,x) \d x \d t \right)^{1/2} \\
&\quad + c\dashint_{I^{\ominus}_{R/2}(t_0)} \hspace{-0.2cm} \tail_{\text{axes}}(u(t);R/2,x_0) \d t,
\end{split}
\end{align}
since we already know that $u$ satisfies the weak Harnack inequality \eqref{eq:wPHI}.

To prove this result, we proceed as in the proof of \eqref{eq:key_improved-lb}, but we use \eqref{eq:cutoff} instead of \eqref{eq:Kupper}. The only place where this replacement is not straightforward is in the estimate of the tail term in \eqref{eq:tail-split}. Here, we proceed as follows:
\begin{align*}
&\sup_{x \in B_{r + \frac{\rho}{2}}(x_0)} \int_{\R^d \setminus B_{r+\rho}(x_0)} |u(s,y)| \mu_{\text{axes}}(x,\d y) \\
&\qquad = \sup_{x \in B_{r + \frac{\rho}{2}}(x_0)} \int_{B_{R}(x_0) \setminus B_{r+\rho}(x_0)} \hspace{-0.2cm} |u(s,y)| \mu_{\text{axes}}(x,\d y) \\
&\qquad\quad + \hspace{-0.2cm} \sup_{x \in B_{r + \frac{\rho}{2}}(x_0)} \int_{\R^d \setminus B_{R}(x_0)} \hspace{-0.2cm} |u(s,y)| \mu_{\text{axes}}(s;x,\d y)\\
&\qquad\le \left(\sup_{B_{R}(x_0)} u(s) \right) \sup_{x \in B_{r + \frac{\rho}{2}}(x_0)} \mu_{\text{axes}}(x,B_{R}(x_0) \setminus B_{r+\rho}(x_0))\\
&\qquad\quad + c R^{-\alpha} \tail_{\text{axes}}(u(s);R,x_0)\\
&\qquad\le c \rho^{-\alpha} \left(\sup_{B_{R}(x_0)} u(s) \right) + c R^{-\alpha} \tail_{\text{axes}}(u(s);R,x_0).
\end{align*}
Note that in the first estimate, we used that the essential supremum of $u(s)$ in $B_R(x_0)$ with respect to $\mu_{\text{axes}}(s)$ coincides with its classical supremum because we assumed that $u$ is bounded and therefore we already know that $x \mapsto u(s,x)$ is also H\"older continuous in $B_R(x_0)$ due to the main result in \cite{KaSc14} (see also \cite{KaWe22a}, \cite{CKW23}, \cite{Wei22thesis}) as a consequence of the weak parabolic Harnack inequality.\\
From here, we continue exactly as in the proof of \autoref{sec:lb_second-proof}, by treating $\sup_{I_R^{\ominus}(t_0) \times B_{R}(x_0)} |u|$ in the same way as the tail term is treated in the well-known proof of the nonlocal De Giorgi iteration (see \cite{KaWe22b}). Thus, we obtain \eqref{eq:key_improved-lb2} with $\tail_{\text{axes}}$ instead of $\tail_{K(t)}$. Again, by a classical covering and iteration argument, we deduce the desired result.
\end{proof}


\end{document}